\documentclass{amsart}
\usepackage[utf8]{inputenc}
\usepackage{amssymb}
\usepackage{hyperref}
\usepackage[final]{showkeys} 

\input xy
\xyoption{all}

\theoremstyle{definition}
\newtheorem{mydef}{Definition}[section]
\newtheorem{lem}[mydef]{Lemma}
\newtheorem{thm}[mydef]{Theorem}
\newtheorem{cor}[mydef]{Corollary}

\newtheorem{hypothesis}[mydef]{Hypothesis}
\newtheorem{prop}[mydef]{Proposition}
\newtheorem{defin}[mydef]{Definition}

\newtheorem{remark}[mydef]{Remark}
\newtheorem{notation}[mydef]{Notation}
\newtheorem{fact}[mydef]{Fact}

\newcommand{\fct}[2]{{}^{#1}#2}

\newcommand{\ba}{\bar{a}}
\newcommand{\bb}{\bar{b}}
\newcommand{\bc}{\bar{c}}

\newcommand{\bx}{\bar{x}}
\newcommand{\by}{\bar{y}}

\newcommand{\bigN}{\widehat{N}}

\newcommand{\bigL}{\widehat{L}}
\newcommand{\bigK}{\widehat{K}}
\newcommand{\bigKp}[1]{\widehat{K}^{<#1}}

\newcommand{\Ksatpp}[2]{{#1}^{#2\text{-sat}}}
\newcommand{\Ksatp}[1]{\Ksatpp{K}{#1}}
\newcommand{\Ksat}{\Ksatp{\lambda}}

\newcommand{\sea}{\mathfrak{C}}
\newcommand{\dom}[1]{\text{dom}(#1)}

\newcommand{\cf}[1]{\text{cf} (#1)}
\newcommand{\seq}[1]{\langle #1 \rangle}
\newcommand{\rest}{\upharpoonright}
\newcommand{\bkappa}{\bar \kappa}

\newcommand{\is}{\mathfrak{i}}

\def\lta{<}
\def\lea{\le}

\def\gea{\ge}

\def\ltu{\lta_{\text{univ}}}

\newcommand{\K}{K}

\newbox\noforkbox \newdimen\forklinewidth
\forklinewidth=0.3pt \setbox0\hbox{$\textstyle\smile$}
\setbox1\hbox to \wd0{\hfil\vrule width \forklinewidth depth-2pt
 height 10pt \hfil}
\wd1=0 cm \setbox\noforkbox\hbox{\lower 2pt\box1\lower
2pt\box0\relax}
\def\unionstick{\mathop{\copy\noforkbox}\limits}
\newcommand{\nf}{\unionstick}

\newcommand{\nfp}[1]{#1\text{-nf}}
\newcommand{\lambdanf}{\nfp{\lambda}}
\newcommand{\chinf}{\nfp{\chi}}

\def\1nf{\unionstick^{(1)}}

\def\2nf{\unionstick^{(2)}}
\def\3nf{\unionstick^{(3)}}

\def\nfm{\unionstick^{\ast}}

\newcommand{\tp}{\text{tp}}
\newcommand{\gtp}{\text{gtp}}
\newcommand{\Ss}{S}
\newcommand{\gS}{\text{gS}}

\newcommand{\hanf}[1]{h (#1)}
\newcommand{\hanfs}[1]{h^\ast (#1)}

\newcommand{\LS}{\text{LS}}

\newcommand{\plus}[1]{{#1}_r}
\newcommand{\kappap}{\plus{\kappa}}
\newcommand{\BI}{\mathbf{I}}
\newcommand{\BJ}{\mathbf{J}}

\newcommand{\Av}{\text{Av}}

\newcommand{\slc}[1]{\bkappa_{#1}}

\newcommand{\slcp}[1]{\slc{#1}^1}

\newcommand{\clc}[1]{\kappa_{#1}}

\title{Chains of saturated models in AECs}
\date{\today}

\author{Will Boney}
\email{wboney@math.harvard.edu}
\urladdr{http://math.harvard.edu/\textasciitilde wboney/}
\address{Department of Mathematics \\ Harvard University \\ Cambridge, Massachusetts, USA}

\author{Sebastien Vasey}
\email{sebv@cmu.edu}
\urladdr{http://math.cmu.edu/\textasciitilde svasey/}
\address{Department of Mathematical Sciences, Carnegie Mellon University, Pittsburgh, Pennsylvania, USA}

\thanks{This material is based upon work done while the first author was supported by the National Science Foundation under Grant No. DMS-1402191 and the second author was supported by the Swiss National Science Foundation under Grant No.\ 155136.}

\begin{document}

\begin{abstract}
We study when a union of saturated models is saturated in the framework of tame abstract elementary classes (AECs) with amalgamation. We prove:

\begin{thm}
  If $K$ is a tame AEC with amalgamation satisfying a natural definition of superstability (which follows from categoricity in a high-enough cardinal), then for all high-enough $\lambda$:
  
  \begin{enumerate}
    \item The union of an increasing chain of $\lambda$-saturated models is $\lambda$-saturated.
    \item There exists a type-full good $\lambda$-frame with underlying class the saturated models of size $\lambda$.
    \item There exists a unique limit model of size $\lambda$.
  \end{enumerate}
\end{thm}

Our proofs use independence calculus and a generalization of averages to this non first-order context.
\end{abstract}

\maketitle

\tableofcontents

\section{Introduction}
Determining when a union of $\lambda$-saturated models is $\lambda$-saturated is an important dividing line in first-order model theory. Recall that Harnik and Shelah have shown:

\begin{fact}[\cite{harniksat}, III.3.11 in \cite{shelahfobook} for the case $\lambda \le |T|$]\label{union-sat-fo} 
  Let $T$ be a first-order theory.
\begin{itemize}
	\item If $T$ is superstable, then any increasing union of $\lambda$-saturated models is $\lambda$-saturated.
	\item If $T$ is stable, then any increasing union of $\lambda$-saturated models of cofinality at least $|T|^+$ is $\lambda$-saturated.
\end{itemize}
\end{fact}

A converse was later proven by Albert and Grossberg \cite[Theorem 13]{agchains}. Fact \ref{union-sat-fo} can be used to prove:

\begin{fact}[The saturation spectrum theorem, VIII.4.7 in \cite{shelahfobook}]\label{sat-spectrum}
  Let $T$ be a stable first-order theory. Then $T$ has a saturated model of size $\lambda$ if and only if [$T$ is stable in $\lambda$ or $\lambda = \lambda^{<\lambda} + |D (T)|$].
\end{fact}

Although not immediately evident from the statement, the proof of Fact \ref{union-sat-fo} relies on the heavy machinery of forking and averages.

While the saturation spectrum theorem has been generalized to homogeneous model theory (see \cite[1.13]{sh54} or \cite[5.9]{grle-homog}), to the best of our knowledge no explicit generalization of Fact \ref{union-sat-fo} has been published in this context (Shelah asserts it without proof in \cite[1.15]{sh54}). Grossberg \cite{grossberg91} has proven a version of Fact \ref{union-sat-fo} in the framework of stability theory inside a model. The proof uses averages but relies on a strong negation of the order property. Makkai and Shelah \cite[4.18]{makkaishelah} have given a generalization in the class of models of an $L_{\kappa, \omega}$ sentence where $\kappa$ is a strongly compact cardinal. The proof uses independence calculus. 

One can ask whether Fact \ref{union-sat-fo} can also be generalized to abstract elementary classes (AECs), a general framework for classification theory introduced in \cite{sh88} (see \cite{grossberg2002} for an introduction to AECs). In \cite[I.5.39]{shelahaecbook}, Shelah proves a generalization of the superstable case of Fact \ref{union-sat-fo} to ``definable-enough'' AECs with countable Löwenheim-Skolem number, using the weak continuum hypothesis.

In chapter II of \cite{shelahaecbook}, Shelah starts with a (weakly successful) good $\lambda$-frame (a local notion of superstability) on an abstract elementary class (AEC) $K$ and wants to show that a union of saturated models is saturated in $K_{\lambda^+}$. For this purpose, he introduces a restriction $\le^{\ast}$ of the ordering that allows him to prove the result for $\le^\ast$-increasing chains (II.7.7 there). Restricting the ordering of the AEC is somewhat artificial and one can ask what happens in the general case, and also if $\lambda^+$ is replaced by an arbitrary cardinal. Moreover, Shelah's methods to obtain a weakly successful good $\lambda$-frame typically use categoricity in two successive cardinals and the weak continuum hypothesis\footnote{See for example \cite[II.3.7]{shelahaecbook}. Shelah also shows how to build a good frame in ZFC from more model-theoretic hypotheses in \cite[IV.4.10]{shelahaecbook}, but he has to change the class and it is not clear his frame is weakly successful.}. 

In \cite{sh394}, Shelah had previously proven that a union of $\lambda$-saturated models is $\lambda$-saturated, for $K$ an AEC with amalgamation, joint embedding, and no maximal models categorical in a successor $\lambda' > \lambda$ (see \cite[Chapter 15]{baldwinbook09} for a writeup), but left the case $\lambda \ge \lambda'$ (or $\lambda'$ not a successor) unexamined.

In this paper, we replace the local model-theoretic assumptions of Shelah with \emph{global} ones, including \emph{tameness}, a locality notion for types introduced by Grossberg and VanDieren \cite{tamenessone}. We take advantage of recent developments in the study of forking in tame AECs (especially by Boney and Grossberg \cite{bg-v11-toappear} and Vasey \cite{sv-infinitary-stability-afml, indep-aec-apal}) to generalize Fact \ref{union-sat-fo} to tame abstract elementary classes with amalgamation. Our main result is:

\textbf{Corollary \ref{cor-ss-aec}.}
  Assume $K$ is a $(<\kappa)$-tame AEC with amalgamation. If $\kappa = \beth_\kappa > \LS (K)$ and $K$ is categorical in some cardinal strictly above $\kappa$, then for all $\lambda > 2^\kappa$, $\Ksat$ (the class of $\lambda$-saturated models of $K$) is an AEC with $\LS (\Ksat) = \lambda$.

Notice that if $\Ksat$ is an AEC, then any increasing union of $\lambda$-saturated models is $\lambda$-saturated. Thus, in contrast to Shelah's \cite{sh394} result, we obtain a \emph{global} theorem that holds for all high-enough $\lambda$ and not just those under the categoricity cardinal. Furthermore categoricity at a successor is not assumed. We can also replace the categoricity by various notions of superstability defined in terms of the local character for independence notions such as coheir or splitting. In fact, we can combine this result with the construction of a good frame in \cite{indep-aec-apal} to obtain the theorem in the abstract:

\textbf{Theorem \ref{ss-def-implication}.}
  If $K$ is a tame AEC with amalgamation satisfying a natural definition of superstability (see Definition \ref{ss-def}), then for all high-enough $\lambda$, there exists a \emph{unique} limit model of size $\lambda$.

This proves an eventual version of a statement appearing in early versions of \cite{gvv-mlq} (see the discussion in Section \ref{ss-def-sec}).

It is very convenient to have $\Ksat$ an AEC, as saturated models are typically better behaved than arbitrary ones. This is crucial for example in Shelah's upward transfer of frames in \cite[Chapter II]{shelahaecbook}, and is also used in \cite{indep-aec-apal} to build an $\omega$-successful good frame (and later a global independence notion). We also prove a result for the strictly stable case: 

\textbf{Theorem \ref{aec-average-stable}.}
  Let $K$ be a $\kappa$-tame AEC with amalgamation, $\kappa \ge \LS (K)$, stable in some cardinal above $\kappa$. Then there exists $\chi_0 \le \lambda_0 < \beth_{(2^{\kappa})^+}$ such that whenever $\lambda \ge \lambda_0$ is such that $\mu^{<\chi_0} < \lambda$ for all $\mu < \lambda$, the union of an increasing chain of $\lambda$-saturated models of cofinality at least $\chi_0$ is $\lambda$-saturated.

One caveat here is the introduction of cardinal arithmetic: in stable, first-order theories, all large enough $\lambda$ satisfy that the long-enough union of $\lambda$-saturated models is $\lambda$-saturated (Fact \ref{union-sat-fo}), while here we have to add the condition that $\lambda$ is $\chi_0$-closed.  When dealing with compact classes (or even just $(<\omega)$-tame classes), the map $\lambda \mapsto \lambda^{<\omega}$ that takes the size of a set to the number of formulas with parameters in that set can be used freely.  Even in the work of Makkai and Shelah \cite{makkaishelah}, where $\kappa$ is strongly compact and the class is $(<\kappa)$-tame, this map is $\lambda \mapsto \lambda^{<\kappa}$, which is constant on most cardinals (those with cofinality at least $\kappa$) by a result of Solovay.  However, in our context of $(<\kappa)$-tameness for $\kappa > \omega$ but not strongly compact, this function can be much wilder (following the Galois Morleyization, we view Galois types of size less than $\kappa$ as formulas; see Definition \ref{def-galois-m}).  Thus, we need to introduce assumptions that this map is well-behaved. Using various tricks, we can bypass these assumptions in the superstable case but are unable to do so in the stable case. For example in Theorem \ref{aec-average-stable}, the cardinal arithmetic assumption can be replaced by ``$K$ is stable in $\mu$ for unboundedly many $\mu < \lambda$'', which is always true in case $K$ is superstable.

We use two main methods: The first method is pure independence calculus, relying on a well-behaved independence relation (coheir), whose existence in our context is proven in \cite{bg-v11-toappear, sv-infinitary-stability-afml}. This works well in the superstable case if we define superstability in terms of coheir (called strong superstability in \cite{indep-aec-apal}) but we do not know how to make it work for weaker definitions of superstability (such as superstability defined in terms of splitting, a more classical definition implicit for example in \cite{gvv-mlq}). The second method is the use of syntactic averages, developed by Shelah in \cite[Chapter V]{shelahaecbook2}. We end up proving a result on chains of saturated models in the framework of stability theory inside a model and then translate to AECs using \emph{Galois Morleyization}, introduced in \cite{sv-infinitary-stability-afml}. This method allows us to use superstability defined in terms of splitting. The two methods give incomparable results: in case we know that $\K$ is $(<\kappa)$-tame, with $\kappa = \beth_\kappa > \LS (\K)$, the first gives better Hanf numbers than the second. However if we know that $\K$ is $\LS (\K)$-tame, then we get better bounds using the second method, since we do not need to work above a fixed point of the beth function.

The paper is organized as follows. Section \ref{stable-sec} gives the argument using independence calculus culminating in Theorems \ref{union-sat-closed} and \ref{union-sat-ss}.  Both of these arguments work just using forking relations, drawing inspiration from Makkai and Shelah, rather than the classical first-order argument using averages.  Section \ref{averages-sec} develops averages in our context based on earlier work of Shelah and culminates in the more local Theorem \ref{union-sat-av}.  Section \ref{mainthm-sec} translates the local result to AECs and Section \ref{ss-def-sec} proves consequences such as the uniqueness of limit models from superstability.

This paper was written while the second author was working on a Ph.D.\ thesis under the direction of Rami Grossberg at Carnegie Mellon University. He would like to thank Professor Grossberg for his guidance and assistance in his research in general and in this work specifically. 

\section{Using independence calculus: the stable case} \label{stable-sec}

We assume that the reader is familiar with the basics of AECs, as presented in \cite{baldwinbook09} or \cite{grossbergbook}. We will use the notation from \cite{sv-infinitary-stability-afml}. In particular, $\gtp (\bb / A; N)$ denotes the Galois type of $\bb$ over $A$ as computed in $N$.

We will use the following set-theoretic notation:

\begin{notation}\label{set-thy-notation}
For $\kappa$ an infinite cardinal, write $\kappap$ for the least regular cardinal greater than or equal to $\kappa$. That is, $\kappap$ is $\kappa^+$ if $\kappa$ is singular and $\kappa$ otherwise. Also let $\kappa^-$ be $\kappa$ if $\kappa$ is limit, or the unique $\kappa_0$ such that $\kappa = \kappa_0^+$ if $\kappa$ is a successor. We will often use the ``Hanf function'' (from \cite[Chapter 14]{baldwinbook09}): for $\lambda$ an infinite cardinal, let $\hanf{\lambda} := \beth_{(2^{\lambda})^+}$. Also let $\hanfs{\lambda} := \hanf{\lambda^-}$.
\end{notation}

All throughout this section, we assume:

\begin{hypothesis}\label{hyp-stab-2} \
  \begin{enumerate}
  \item $\K$ is an AEC with amalgamation, joint embedding, and arbitrarily large models. We work inside a monster model $\sea$.
  \item $\LS (\K) < \kappa = \beth_\kappa$.
  \item $\K$ is $(<\kappa)$-tame.
  \item $\K$ is stable (in some cardinal above $\kappa$).
  \end{enumerate}
\end{hypothesis}

We will use the independence notion of coheir for AECs, introduced in \cite{bg-v11-toappear}.

\begin{defin}[Coheir]
  Define a tertiary relation $\nf$ by $\nf (M, A, B)$ if and only if:
  \begin{enumerate}
  \item $M \lea \sea$, $M$ is $\kappa$-saturated, and $A, B \subseteq |\sea|$.
  \item For any $\ba \in \fct{<\kappa}{A}$ and $B_0 \subseteq |M| \cup B$ of size less than $\kappa$, there exists $\ba' \in \fct{<\kappa}{|M|}$ such that $\gtp (\ba / B_0) = \gtp (\ba' / B_0)$ (here, the Galois types are computed inside $\sea$).
  \end{enumerate}

  We write $A \nf_M B$ instead of $\nf (M, A, B)$. We will also say that $\gtp (\ba / B)$ is a \emph{$(<\kappa)$-coheir over $M$} when ${\operatorname{ran} (\ba)} \nf_M B$ (it is straightforward to check that this does not depend on the choice of $\ba$).
\end{defin}

The following locality cardinals will play an important role:

\begin{defin}\label{loc-card-def}
  Let $\alpha$ be a cardinal.
  
  \begin{enumerate}
  \item Let $\slc{\alpha} (\nf)$ be the minimal cardinal $\mu \ge |\alpha|^+ + \kappa^+$ such that for any $M \lea \sea$ that is $\kappa$-saturated, any $A \subseteq |\sea|$ with $|A| \le \alpha$, there exists $M_0 \lea M$ in $K_{<\mu}$ with $A \nf_{M_0} M$. When $\mu$ does not exist, we set $\slc{\alpha_0} (\is) = \infty$.

  \item Let $\clc{\alpha} (\nf)$ be the minimal cardinal $\mu \ge |\alpha|^+ + \aleph_0$ such that for any regular $\delta \ge \mu$, any increasing chain $\seq{M_i : i < \delta}$ in $K$ and any $A$ of size at most $\alpha$, there exists $i < \delta$ such that $A \nf_{M_i} \bigcup_{i < \delta} M_i$. When $\mu$ does not exist, we set $\clc{\alpha_0} (\is) = \infty$. For $\K^\ast$ a subclass of $\K$, we similarly define $\clc{\alpha} (\nf \rest K^\ast)$, where in addition we require that $M_i \in \K^\ast$ for all $i < \delta$ (we will use this when $\K^\ast$ is a class of saturated models).
  \end{enumerate}
\end{defin}
\begin{remark}\label{loc-card-rmk}
  For any cardinal $\alpha$, we always have that $\clc{\alpha} (\nf) \le \slc{\alpha} (\nf)$.
\end{remark}

\begin{fact}\label{coheir-props-fact}
  Under Hypothesis \ref{hyp-stab-2}, $\nf$ satisfies the following properties:

  \begin{enumerate}
  \item \underline{Invariance}: If $f$ is an automorphism of $\sea$ and $A \nf_M B$, then $f[A] \nf_{f[M]} B$.
  \item \underline{Monotonicity}: Assume $A \nf_M B$, then:
    \begin{enumerate}
    \item Left and right monotonicity: If $A_0 \subseteq A$, $B_0 \subseteq B$, then $A_0 \nf_M B_0$.
    \item Base monotonicity: If $M \lea M' \lea \sea$, $|M'| \subseteq B$, and $M'$ is $\kappa$-saturated, then $A \nf_{M'} B$.
    \end{enumerate}
  \item \underline{Left and right normality}: If $A \nf_M B$, then $A M \nf_M BM$.
  \item \underline{Symmetry}: $A \nf_M B$ if and only if $B \nf_M A$.
  \item \underline{Strong transitivity}: If $M_0 \lea \sea$, $M_1 \lea \sea$, $A \nf_{M_0} M_1$, and $A \nf_{M_1} B$, then $A \nf_{M_0} B$ (note that we do \emph{not} assume that $M_0 \lea M_1$).
  \item \underline{Uniqueness for types of length one}: If $M \lea M'$, $p, q \in \gS (M')$ are both $(<\kappa)$-coheir over $M$ and $p \rest M = q \rest M$, then $p = q$.
  \item \underline{Set local character}: For any $\alpha$, $\slc{\alpha} (\nf) \le \left((\alpha + 2)^{<\kappap}\right)^+$.
  \end{enumerate}

  Moreover $\K$ is stable in all $\mu \ge \kappa$ such that $\mu = \mu^{2^{<\kappap}}$.
\end{fact}
\begin{remark}
  We will not use the exact definition of coheir, just that it satisfies the conclusion of Fact \ref{coheir-props-fact}. 
\end{remark}
\begin{remark}
  Strong transitivity will be used in the proof that the relation $\nfm$ (Definition \ref{nfm-def}) is transitive, see Proposition \ref{nfm-basic}. We do not know if transitivity would suffice.
\end{remark}

For what comes next, it will be convenient if we could say that $A \nf_M B$ and $M \lea N$ implies $A \nf_N B$. By base monotonicity, this holds if $|N| \subseteq B$ but in general this is not part of our assumptions (and in practice this need not hold). Thus we will close $\nf$ under this property. This is where we depart from \cite{makkaishelah}; there the authors used that the singular cardinal hypothesis holds above a strongly compact to prove the result corresponding to our Lemma \ref{local-character-finite}. Here we need to be more clever.

\begin{defin}\label{nfm-def}
  $A \nfm_C B$ means that there exists $M_0 \lea \sea$, $|M_0| \subseteq C$ such that $A \nf_{M_0} B$.
\end{defin}
\begin{remark}
  $\nfm$ need not satisfy the normality property from Fact \ref{coheir-props-fact}.
\end{remark}

In what follows,  we will apply the definition of $\clc{\alpha}$ and $\slc{\alpha}$ (Definition \ref{loc-card-def}) to other independence relations than coheir.

\begin{defin}
  We write $A \nfm_C [B]^1$ to mean that $A \nfm_C b$ for all $b \in B$. Similarly define $[A]^1 \nfm_C B$. Let $(\nfm)^1$ denote the relation defined by $A (\nfm_C)^1 B$ if and only if $A \nfm_C [B]^1$. For $\alpha$ a cardinal, let $\slcp{\alpha} = \slcp{\alpha} (\nfm) := \slc{\alpha} ((\nfm)^1)$.
\end{defin}

Note that $A \nfm_C B$ implies $[A]^1 \nfm_C B$ by monotonicity.

\begin{prop}\label{nfm-basic} \
  \begin{enumerate}
    \item $\nfm$ satisfies invariance, monotonicity, symmetry, and strong right transitivity (see Fact \ref{coheir-props-fact}). 
    \item\label{nfm-6} For all $\alpha$, $\slc{\alpha} (\nfm) = \slc{\alpha} (\nf)$, $\clc{\alpha} (\nfm) = \clc{\alpha} (\nf)$.
    \item $\nfm$ has \emph{strong base monotonicity}: If $A \nfm_C B$ and $C \subseteq C'$, then $A \nfm_{C'} B$.
    \item If $A \nf_M B$, then $A \nfm_M B$.
    \item\label{nfm-4} If $A \nfm_M B$ and $M$ is $\kappa$-saturated such $|M| \subseteq B$, then $A \nf_M B$.
    \item\label{nfm-5} For all $\alpha$, $\slcp{\alpha}(\nfm) \le \slc{\alpha} (\nf)$.
  \end{enumerate}
\end{prop}
\begin{proof}
  All quickly follow from the definition. As an example, we prove that $\nfm$ has strong right transitivity. Assume $A \nfm_{M_0} M_1$ and $A \nfm_{M_1} B$. Then there exists $M_0' \lea M_0$ and $M_1' \lea M_1$ such that $A \nf_{M_0'} M_1$ and $A \nf_{M_1'} B$. By monotonicity for $\nf$, $A \nf_{M_0'} M_1'$. By strong right transitivity for $\nf$, $A \nf_{M_0'} B$. Thus $M_0'$ witnesses $A \nfm_{M_0} B$.
\end{proof}

\begin{prop}\label{cont-prop}
  Assume $\seq{M_i : i < \delta}$, $\seq{N_i :i < \delta}$ are increasing chains of $\kappa$-saturated models, $A$ is a set. If $A \nfm_{M_i} N_i$ for all $i < \delta$ and $\clc{|A|} (\nf) \le \cf{\delta}$, then $A \nfm_{M_\delta} N_\delta$, where\footnote{Note that $M_\delta$ and $N_\delta$ need not be $\kappa$-saturated.} $M_\delta := \bigcup_{i < \delta} M_i$ and $N_\delta := \bigcup_{i < \delta} N_i$.
\end{prop}
\begin{proof}
  Without loss of generality, $\delta = \cf{\delta}$. By definition of $\clc{|A|} (\nf)$, there exists $i < \delta$ such that $A \nf_{N_i} N_\delta$, so $A \nfm_{N_i} N_\delta$. By strong right transitivity for $\nfm$, $A \nfm_{M_i} N_\delta$. By strong base monotonicity, $A \nfm_{M_\delta} N_\delta$.
\end{proof}

As already discussed, the reason we use $\nfm$ is that we want to generalize \cite[4.17]{makkaishelah} to our context. In their proof, Makkai and Shelah use that cardinal arithmetic behaves nicely above a strongly compact, and we cannot make use of this fact here. Thus we are only able to prove this lemma for $\nfm$ instead of $\nf$ (see Lemma \ref{local-character-finite}). Fortunately, this turns out to be enough. The reader can also think of $\nfm$ as a trick to absorb some quantifiers.

The next lemma imitates \cite[4.18]{makkaishelah}.

\begin{lem}\label{union-sat-2}
  Let $\lambda_0 \ge \kappap$ be regular, let  $\lambda > \lambda_0$ be regular such that $\K$ is stable in unboundedly-many cardinals below $\lambda$ and let $\seq{M_i : i < \delta}$ be an increasing chain with $M_i$ $\lambda$-saturated for all $i < \delta$. Assume that $\clc{1} (\nf \rest \Ksatp{\lambda_0}) \le \cf{\delta}$.
  
  If $\slcp{<\lambda} (\nfm) \le \lambda$, then $M_\delta := \bigcup_{i < \delta} M_i$ is $\lambda$-saturated.
\end{lem}
\begin{proof}
  Without loss of generality, $\delta = \cf{\delta}$.
  Let $A \subseteq |M_\delta|$ have size less than $\lambda$. If $\lambda \le \delta$, then $A \subseteq |M_i|$ for some $i < \delta$ and so any type over $A$ is realized in $M_i \subseteq |M_\delta|$. Now assume without loss of generality that $\lambda > \delta$. We need to show every Galois type over $A$ is realized in $M_\delta$. Let $\mu := \lambda_0 + \delta$.  Note that $\mu = \cf{\mu} < \lambda$. First, we build an array of $\lambda_0$-saturated models $\seq{N_i^\alpha \in K_{<\lambda} : i < \delta, \alpha < \mu}$ such that:

\begin{enumerate}
  \item For all $i < \delta$, $\seq{N_i^\alpha : \alpha < \mu}$ is increasing.
  \item For all $\alpha < \mu$, $\seq{N_i^\alpha : i < \delta}$ is increasing.
  \item For all $i < \delta$ and all $\alpha < \mu$, $N_i^\alpha \le M_i$.
  \item\label{cond-5} $A \subseteq \bigcup_{i < \delta} |N_i^0|$.
  \item\label{cond-6} For all $\alpha < \mu$ and all $i < \delta$, $\bigcup_{i < \delta} N_i^\alpha \nfm_{N_i^{\alpha + 1}} [M_i]^1$.
\end{enumerate}

For $\alpha < \mu$, write $N_\delta^\alpha := \bigcup_{i < \delta} N_i^\alpha$ and for $i \le \delta$, write $N_i^\mu := \bigcup_{\alpha < \mu} N_i^\alpha$. The following is a picture of the array constructed.

\[
\xymatrix{
M_i \ar[r] & M_{i+1} \ar[rr] &  & M_\delta \\
 &  &  &  \\
N_i^\mu \ar[r] \ar[uu] & N_{i+1}^\mu \ar[uu] \ar[rr] &  & N_\delta^\mu \ar[uu] \\
N_i^{\alpha+1} \ar[r] \ar[u] & N_{i+1}^{\alpha+1} \ar[rr] \ar[u] &  & N_\delta^{\alpha+1} \ar[u]\\
N_i^\alpha \ar[r] \ar[u] & N_{i+1}^\alpha \ar[rr] \ar[u] &  & N_\delta^\alpha \ar[u]\\
 &  &  & A \ar[u] 
}
\]

\underline{This is enough}: Note that for $i < \delta$, $N_i^\mu$ is $\lambda_0$-saturated and has size less than $\lambda$ (since $\lambda > \mu$ and $\lambda$ is regular). Note also that since $\delta \le \mu < \lambda$, $N_\delta^\mu$ has size less than $\lambda$ (but we do \emph{not} claim that it is $\lambda_0$-saturated).

\textbf{Claim:} For all $i < \delta$, $N_\delta^\mu \nfm_{N_i^\mu} [M_i]^1$. \\
\textbf{Proof of claim:} Fix $i < \delta$ and let $a \in M_i$. Fix $j < \delta$. By (\ref{cond-6}), monotonicity, and symmetry, $a \nfm_{N_i^{\alpha + 1}} N_j^\alpha$ for all $\alpha < \mu$. By Proposition \ref{cont-prop} applied to $\seq{N_i^{\alpha + 1} : \alpha < \mu}$ and $\seq{N_j^\alpha : \alpha < \mu}$, $a \nfm_{N_i^\mu} N_j^\mu$ (note that $\mu = \cf{\mu} \ge \delta \ge \clc{1} (\nf)$). Since $j$ was arbitrary, we can apply Proposition \ref{cont-prop} again with the constantly $N_i^\mu$ sequence and $\seq{N_j^\mu : j < \delta}$ (note that $\delta = \cf{\delta} \ge \clc{1} (\nf)$) to get that $a \nfm_{N_i^\mu} N_\delta^\mu$. By symmetry, $N_i^\mu \nfm_{N_i^\mu} a$, as desired.\hfill $\dag_{Claim}$\\

Now let $p \in \gS (A)$. By (\ref{cond-5}), $A \subseteq N_\delta^\mu$ so we can extend $p$ to some $q \in \gS (N_\delta^\mu)$. Since $\delta \ge \clc{1} (\nf)$, we can find $i < \delta$ such that $q$ does not fork over $N_i^\mu$. Since $N_i^\mu \le M_i$, $M_i$ is $\lambda$-saturated, and $\|N_i^\mu\| < \lambda$, we can find $a \in M_i$ realizing $q \rest N_i^\mu$. Since by the claim $N_\delta^\mu \nfm_{N_i^\mu} [M_i]^1$, we can use symmetry to conclude $a \nfm_{N_i^\mu} N_\delta^\mu$, and hence (Proposition \ref{nfm-basic}(\ref{nfm-4})) $a \nf_{N_i^\mu} N_\delta^\mu$. By uniqueness for types of length one, $a$ must realize $q$, so in particular $a$ realizes $p$. This concludes the proof that $M_\delta$ is $\lambda$-saturated.

\underline{This is possible}: We define $\seq{N_i^\alpha : i < \delta}$ by induction on $\alpha$. For a fixed $i < \delta$, choose any $N_i^0 \lea M_i$ in $K_{<\lambda}$ that contains $A \cap |M_i|$ and is $\lambda_0$-saturated (this is possible since $\K$ is stable in unboundedly-many cardinals below $\lambda$). For $\alpha < \mu$ limit and $i < \delta$, pick any $N_i^\alpha \lea M_i$ containing $\bigcup_{\beta < \alpha} N_i^\beta$ which is in $K_{<\lambda}$ and $\lambda_0$-saturated (this is possible for the same reason as in the base case). Now assume $\alpha = \beta + 1 < \mu$, and $N_i^\beta$ has been defined for $i < \delta$. Define $N_i^{\alpha}$ by induction on $i$. Assume $N_j^{\alpha}$ has been defined for all $j < i$. Pick $N_i^{\alpha}$ containing $\bigcup_{j < i} N_j^{\alpha}$ which is in $K_{<\lambda}$, is $\lambda_0$-saturated, and satisfies $N_i^\alpha \lea M_i$ and (\ref{cond-6}). This is possible by strong base monotonicity and definition of $\slcp{<\lambda}$.
\end{proof}

Below, we give a more natural formulation of the hypotheses. 

\begin{thm}\label{union-sat-closed}
  Let $\lambda > \kappa$. Let $\seq{M_i : i < \delta}$ be an increasing chain with $M_i$ $\lambda$-saturated for all $i < \delta$. If:

  \begin{enumerate}
    \item $\cf{\delta} \ge \clc{1} (\nf)$; and
    \item\label{union-sat-arithm} $\chi^{2^{<\kappap}}< \lambda$ for all $\chi < \lambda$,
  \end{enumerate}

  then $\bigcup_{i < \delta} M_i$ is $\lambda$-saturated.
\end{thm}
\begin{proof}
  Let $M_\delta := \bigcup_{i < \delta} M_i$. Note that $\lambda > \kappap$: since $\lambda > \kappa$, $\lambda \ge \kappa^+$ and if $\lambda = \kappa^+$ then $\kappa^{<\kappa} < \lambda$ so $\kappa = \kappa^{<\kappa}$ hence $\kappa$ is regular: $\kappap = \kappa$.

 Let $\chi < \lambda$ be such that $\chi^+ > \kappap$. We show that $M_\delta$ is $\chi^+$-saturated. By hypothesis, $\chi^{2^{<\kappap}} < \lambda$, so replacing $\chi$ by $\chi^{<\kappap}$ if necessary, we might as well assume that $\chi = \chi^{2^{<\kappap}}$. We check that $\chi^+$ satisfies the conditions of Lemma \ref{union-sat-2} (with $\lambda_0$ there standing for $\kappap$ here) as $\lambda$ there. By assumption, $\chi^+$ is regular and $\chi^+ > \kappap$. Also, $\K$ is stable in unboundedly-many cardinals below $\chi^+$ because by the moreover part of Fact \ref{coheir-props-fact}, $\K$ is stable in $\chi$.

 Now by Proposition \ref{nfm-basic}(\ref{nfm-5}), $\slcp{\chi} (\nfm)
 \le \slc{\chi} (\nf)$. By Fact \ref{coheir-props-fact}, $\slc{\chi} (\nf) \le (\chi^{<\kappap})^+ = \chi^+$. Thus $\slcp{\chi} (\nfm) \le \chi^+$, as needed.

  Thus Lemma \ref{union-sat-2} applies and so $M_\delta$ is $\chi^+$-saturated. Since $\chi < \lambda$ was arbitrary, $M_\delta$ is $\lambda$-saturated.
\end{proof}

For the next corollaries to AECs, we repeat our hypotheses.

\begin{cor}\label{indep-stable}
  Let $\K$ be an AEC with amalgamation. Let $\kappa = \beth_\kappa > \LS (\K)$ be such that $\K$ is $(<\kappa)$-tame. Assume that $\K$ is stable in some cardinal greater than or equal to $\kappa$ and let $\seq{M_i : i < \delta}$ be an increasing chain of $\lambda$-saturated models. If:

  \begin{enumerate}
    \item $\cf{\delta} > 2^{<\kappap}$.
    \item\label{union-sat-arithm-2} $\chi^{2^{<\kappap}} < \lambda$ for all $\chi < \lambda$.
  \end{enumerate}
  
  Then $\bigcup_{i < \delta} M_i$ is $\lambda$-saturated.
\end{cor}
\begin{proof}
  Without loss of generality, $\delta = \cf{\delta} < \lambda$. Also without loss of generality, $\K$ has joint embedding (otherwise, partition it into disjoint classes, each of which has joint embedding), and arbitrarily large models (since $\K$ has a model of cardinality $\kappa = \beth_\kappa > \LS (\K)$). Therefore Hypothesis \ref{hyp-stab-2}, and hence the conclusion of Fact \ref{coheir-props-fact}, hold.
  
Note (Remark \ref{loc-card-rmk}) that $\clc{1} (\nf) \le \slc{1} (\nf) \le \left(2^{<\kappap}\right)^+$. Now use Theorem \ref{union-sat-closed}.
\end{proof}


  

\section{Using independence calculus: the superstable case}

Next we show that in the \emph{superstable} case we can remove the cardinal arithmetic condition (\ref{union-sat-arithm-2}) in Corollary \ref{indep-stable}. 

\begin{hypothesis}
  Same as in the previous section: Hypothesis \ref{hyp-stab-2}.
\end{hypothesis}

In the proof of Theorem \ref{union-sat-closed}, we estimated $\slcp{\alpha} (\nf)$ using $\slc{\alpha} (\nf)$. Using superstability, we can prove a better bound. This is adapted from \cite[4.17]{makkaishelah}.

\begin{lem}\label{local-character-finite}
  Assume that $\clc{1} (\nf) = \aleph_0$ and $\K$ is stable in all $\lambda \ge \kappap$. Then for any cardinal $\alpha$, $\slcp{\alpha} (\nfm) \le \slc{\kappap} (\nf) + \alpha^+$.
\end{lem}
\begin{proof}
  Let $A$ have size $\alpha$ and $N$ be a $\kappa$-saturated model. We show by induction on $\alpha$ that there exists an $M \lea N$ with $\|M\| < \mu := \slc{\kappap} (\nf) + \alpha^+$ and $A \nfm_M [N]^1$. Note that $\mu > \kappap$.

  If $\alpha \le \kappap$, then apply the definition of $\slc{\kappap} (\nf)$ to get a $M \le N$ with $\|M\| < \slc{\kappap} (\nf)$, $A \nfm_M N$, which is more than what we need.

  Now, assume $\alpha > \kappap$, and that the result has been proven for all $\alpha_0 < \alpha$. Closing $A$ to a $\kappa$-saturated model (using the stability assumptions) if necessary, we can assume without loss of generality that $A$ is a $\kappa$-saturated model. Let $\seq{A_i : i < \alpha}$ be an increasing resolution of $A$ such that $A_i$ is $\kappa$-saturated in $\K_{<\alpha}$ for all $i < \alpha$. Now define an increasing chain $\seq{M_i :i < \alpha}$ such that for all $i < \alpha$:

  \begin{enumerate}
  \item $M_i \in K_{<\mu}$ and $M_i$ is $\kappa$-saturated.
  \item $M_i \lea N$.
  \item $A_i \nfm_{M_i} [N]^1$.
  \end{enumerate}

  \underline{This is possible}: For $i < \alpha$, use the induction hypothesis to find $M_i \lea N$ such that $A_i \nfm_{M_i} [N]^1$ and $\|M_i\| < \mu$. By strong base monotonicity of $\nfm$ and the closure assumption, we can assume that $M_i$ contains $\bigcup_{j < i} M_j$.

  \underline{This is enough}: Let $M \in K_{<\mu}$ be $\kappa$-saturated and contain $\bigcup_{i < \alpha} M_i$. We claim that $A \nfm_M [N]^1$. Let $a \in N$. By symmetry, it is enough to see $a \nfm_M A$. This follows from strong base monotonicity and Proposition \ref{cont-prop} applied to $\seq{M_i : i < \alpha}$ and $\seq{A_i : i < \alpha}$ since $\clc{1}(\nfm) = \aleph_0 \le \cf{\alpha}$ by Proposition \ref{nfm-basic}(\ref{nfm-6}) and the hypothesis.
\end{proof}
\begin{remark}
The heavy use of strong base monotonicity in the above proof was the reason for introducing $\nfm$. 
\end{remark}

\begin{thm}\label{union-sat-ss}
  Let $\lambda_0 \ge \kappap$ be regular. Assume that $\clc{1} (\nf \rest \Ksatp{\lambda_0}) = \aleph_0$ and $K$ is stable in all $\lambda \ge \lambda_0$. Let $\lambda \ge \slc{\kappap} (\nf) + \lambda_0^+$.

  Let $\seq{M_i : i < \delta}$ be an increasing chain with $M_i$ $\lambda$-saturated for all $i < \delta$. Then $M_\delta := \bigcup_{i < \delta} M_i$ is $\lambda$-saturated.
\end{thm}
\begin{proof}
  Let $\chi < \lambda$ be such that $\chi^+ \ge \slc{\kappap} (\nf) + \lambda_0^+$. We claim that $\chi^+$ satisfies the hypotheses of Lemma \ref{union-sat-2} (as $\lambda$ there). Indeed by Lemma \ref{local-character-finite}, $\slcp{\chi} (\nfm) \le \slc{\kappap} (\nf) + \chi^+ = \chi^+$.

  Thus Lemma \ref{union-sat-2} applies: $M_\delta$ is $\chi^+$-saturated. Since $\chi < \lambda$ was arbitrary, $M_\delta$ is $\lambda$-saturated.
\end{proof}

For the next corollary to AECs, we drop our hypotheses.

\begin{cor}\label{cor-ss-aec}
  Let $\K$ be an AEC with amalgamation and no maximal models. Let $\kappa = \beth_\kappa > \LS (\K)$ be such that $\K$ is $(<\kappa)$-tame. If $\K$ is categorical in some cardinal strictly above $\kappa$, then for all $\lambda > 2^\kappa$, $\Ksatp{\lambda}$ is an AEC with Löwenheim-Skolem number $\lambda$.
\end{cor}
\begin{proof}
  Using categoricity and amalgamation, it is easy to check that $\K$ has joint embedding. Let $\lambda_0 := \kappa^+$. By \cite[10.8,10.16]{indep-aec-apal}, $\K$ is stable in all $\mu \ge \kappa$ and $\clc{1} (\nf \rest \Ksatp{\lambda_0}) = \aleph_0$. In particular, Hypothesis \ref{hyp-stab-2} holds. Remembering (Fact \ref{coheir-props-fact}) that $\slc{\kappap} (\nf) \le \kappap^{<\kappap} \le 2^{\kappa}$, we obtain the result from Theorem \ref{union-sat-ss} (to show that $\LS (\Ksatp{\lambda}) = \lambda$, imitate the proof of \cite[III.3.12]{shelahfobook}).
\end{proof}

\section{Averages} \label{averages-sec}

In this section, we write in the framework of stability theory inside a model:

\begin{hypothesis} \
  \begin{enumerate}
  \item $\kappa$ is an infinite cardinal.
  \item $L$ is a $(<\kappa)$-ary language.
  \item $\mathcal{N}$ is a fixed $L$-structure.
  \item We work inside $\mathcal{N}$.
  \item Hypotheses \ref{model-hyp} and \ref{sat-hyp}, see the discussion below.
  \end{enumerate}
  
  Midway through, we will also assume Hypothesis \ref{midway-hyp}.
\end{hypothesis}

We use the same notation and convention as \cite[Section 2]{sv-infinitary-stability-afml}: although we may forget to say it, we always work with \emph{quantifier-free} $L_{\kappa, \kappa}$ formulas and types (so the arity of all the variables inside a given formula is less than $\kappa$). Also, since we work inside $\mathcal{N}$, everything is defined relative to $\mathcal{N}$. For example $\tp (\bc / A)$ means $\tp_{qL_{\kappa, \kappa}} (\bc / A; \mathcal{N})$, the \emph{quantifier-free} $L_{\kappa, \kappa}$-type of $\bc$ over $A$, and saturated means saturated in $\mathcal{N}$. Similarly, we write $\models \phi[\bb]$ instead of $\mathcal{N} \models \phi[\bb]$. By ``type'', we mean a member of $\Ss^{<\infty} (A)$ for some set $A$. Whenever we mention a set of formulas (meaning a possibly incomplete type), we mean a ($L_{\kappa, \kappa}$-quantifier free) set of formulas that is satisfiable by an element in $\mathcal{N}$.

Unless otherwise noted, the letters $\ba$, $\bb$, $\bc$ denote tuples of elements \emph{of length less than $\kappa$}. The letters $A$, $B$, $C$, will denote subsets of $\mathcal{N}$. We say $\seq{A_i : i < \delta}$ is \emph{increasing} if $A_i \subseteq A_j$ for all $i < j < \delta$.

We say $\mathcal{N}$ is \emph{$\alpha$-stable in $\lambda$} if $|\Ss^\alpha (A)| \le \lambda$ for all $A$ with $|A| \le \lambda$ (the default value is for $\alpha$ is $1$). We say $\mathcal{N}$ has the \emph{order property of length $\chi$} if there exists a (quantifier-free) formula $\phi (\bx, \by)$ and elements $\seq{\ba_i : i < \chi}$ of the same arity (less than $\kappa$) such that for $i, j < \chi$, $\models \phi[\ba_i, \ba_j]$ if and only if $i < j$. 

Boldface letters like $\BI$, $\BJ$ will always denote \emph{sequences of tuples of the same arity} (less than $\kappa$). We will use the corresponding non-boldface letter to denote the linear ordering indexing the sequence (writing for example $\BI = \seq{\ba_i : i \in I}$, where $I$ is a linear order). We sometimes treat such sequences as sets of tuples, writing statements like $\ba \in \BI$, but then we are really looking at the range of the sequence. To avoid potential mistakes, we do not necessarily assume that the elements of $\BI$ are all distinct although it should always hold in cases of interest. We write $|\BI|$ for the cardinality of the range, i.e.\ the number of distinct elements in $\BI$. We will sometimes use the interval notation on linear order. For example, if $I$ is a linear order and $i \in I$, $[i, \infty)_I := \{j \in I \mid j \ge i\}$.

As the reader will see, this section builds on earlier work of Shelah from \cite[Chapter V.A]{shelahaecbook2}. Note that Shelah works in an arbitrary logic. We work only with quantifier-free $L_{\kappa, \kappa}$-formulas in order to be concrete and because this is the case we are interested in to translate the syntactic results to AECs.

The reader may wonder what the right notion of submodel is in this context. We could simply say that it is ``subset'' but this does not quite work when translating to AECs. Thus we fix a set of subsets of $\mathcal{N}$ that by definition will be the substructures of $\mathcal{N}$. We require that this set satisfies some axioms akin to those of AECs. This can be taken to be the full powerset if one is not interested in doing an AEC translation.  

\begin{hypothesis}\label{model-hyp}
  $\mathcal{S} \subseteq \mathcal{P} (|\mathcal{N}|)$ is a fixed set of subsets of $\mathcal{N}$ satisfying:

  \begin{enumerate}
    \item Closure under chains: If $\seq{A_i : i < \delta}$ is an increasing sequence of members of $\mathcal{S}$, then $\bigcup_{i < \delta} A_i$ is in $\mathcal{S}$.
    \item Löwenheim-Skolem axiom: If $A \subseteq B$ are sets and $B \in \mathcal{S}$, there exists $A' \in \mathcal{S}$ such that $A \subseteq A' \subseteq B$ and $|A'| \le (|L| + 2)^{<\kappa} + |A|$.
  \end{enumerate}

  We exclusively use the letters $M$ and $N$ to denote elements of $\mathcal{S}$ and call such elements \emph{models}. We pretend they are $L$-structures and write $|M|$ and $|N|$ for their universe and $\|M\|$ and $\|N\|$ for their cardinalities.
\end{hypothesis}
\begin{remark}
  An element $M$ of $\mathcal{S}$ is \emph{not} required to be an $L$-structure. Note however that if it is $\kappa$-saturated for types of length less than $\kappa$ (see below), then it will be one.
\end{remark}

We also need to discuss the definition of saturated: define $M$ to be \emph{$\lambda$-saturated for types of length $\alpha$} if for any $A \subseteq |M|$ of size less than $\lambda$, any $p \in \Ss^\alpha (A)$ is realized in $M$. Similarly define $\lambda$-saturated for types of length less than $\alpha$.  Now in the framework we are working in, $\mu$-saturated for types of length less than $\kappa$ seems to be the right notion, so we say that \emph{$M$ is $\mu$-saturated} if it is $\mu$-saturated for types of length less than $\kappa$. Unfortunately it is not clear that it is equivalent to $\mu$-saturated for types of length one (or length less than $\omega$), even when $\mu > \kappa$. However \cite[II.1.14]{shelahaecbook} (the ``model-homogeneity = saturativity'' lemma) tells us that in case $\mathcal{N}$ comes from an AEC, then this is the case. Thus we will make the following additional assumption. Note that it is possible to work without it, but then everywhere below ``stability'' must be replaced by ``$(<\kappa)$-stability''.

\begin{hypothesis}\label{sat-hyp}
  If $\mu > (|L| + 2)^{<\kappa}$, then whenever $M$ is $\mu$-saturated for types of length one, it is $\mu$-saturated (for types of length less than $\kappa$).
\end{hypothesis}

Our goal in this section is to use Shelah's notion of average in this framework to prove a result about chains of saturated models. Recall:

\begin{defin}[Definition V.A.2.6 in \cite{shelahaecbook2}]
  For $\BI$ a sequence, $\chi$ an infinite cardinal such that $|\BI| \ge \chi$, and $A$ a set, define $\Av_{\chi} (\BI / A)$ to be the set of formulas $\phi (\bx)$ over $A$ so that the set $\{\bb \in \BI \mid \models \neg \phi[\bb]\}$ has size less than $\chi$.
\end{defin}

Note that if $|\BI| \ge \chi$ (say all the elements of $\BI$ have the same arity $\alpha$) and $\phi (\bx)$ is a formula with $\ell (\bx) = \alpha$, then \emph{at most one} of $\phi$, $\neg \phi$ is in $\Av_{\chi} (\BI / A)$. Thus the average is not obviously contradictory, but we do \emph{not} claim that there is an element in $\mathcal{N}$ realizing it.  Also, $\Av_\chi (\BI/A)$ might be empty.  However, we give conditions below (see Fact \ref{conv-thm} and Theorem \ref{coherence-convergence}) where it is in fact complete (i.e.\ exactly one of $\phi$ and $\neg \phi$ is in the average).

  The next lemma is a simple counting argument allowing us to find such an element:

\begin{lem}\label{average-realize}
  Let $\BI$ be a sequence with $|\BI| \ge \chi$ and let $A$ be a set. Let $p := \Av_{\chi} (\BI / A)$. Assume that

  $$
  |\BI| > \chi + \min ((|A| + |L| + 2)^{<\kappa}, |\Ss^{\ell (p)} (A)|)
  $$

  Then there exists $\bb \in \BI$ realizing $p$.
\end{lem}
\begin{proof}
  Assume first that the minimum is realized by $(|A| + |L| + 2)^{<\kappa}$. By definition of the average, for every every formula $\phi (\bx) \in p$, $\BJ_{\phi} := \{\bb \in \BI \mid \models \neg \phi[\bb]\}$ has size less than $\chi$. Let $\BJ := \bigcup_{\phi \in p} \BJ_{\phi}$. Note that $|\BJ| \le \chi + (|A| + |L| + 2)^{<\kappa}$ and by definition any $\bb \in \BI \backslash \BJ$ realizes $p$.

  Now assume that the minimum is realized by $|\Ss^{\ell (p)} (A)|$. Let $\mu := \chi + |\Ss^{\ell (p)} (A)|$. By the pigeonhole principle, there exists $\BI_0 \subseteq \BI$ of size $\mu^+$ such that $\bc, \bc' \in \BI_0$ implies $q := \tp (\bc / A) = \tp (\bc' / A)$. We claim that $p \subseteq q$, which is enough: any $\bb \in \BI_0$ realizes $p$. If not, there exists $\phi (\bx) \in p$ such that $\neg \phi (\bx) \in q$. By definition of the average, fewer than $\chi$-many elements of $\BI$ satisfy $\neg \phi (\bx)$. However, $\neg \phi (\bx)$ is in $q$ which means that it is realized by all the elements of $\BI_0$ and $|\BI_0| = \mu^+ > \chi$, a contradiction.
\end{proof}

We now recall the definition of splitting and study how it interacts with averages.

\begin{defin}
  A set of formulas $p$ \emph{splits over $A$} if there exists $\phi (\bx, \bb) \in p$ and $\bb'$ with $\tp (\bb' / A) = \tp (\bb / A)$ and $\neg \phi (\bx, \bb') \in p$.
\end{defin}

The following result is classical:

\begin{lem}[Uniqueness for nonsplitting]\label{ns-uq}
  Let $A \subseteq |M| \subseteq B$. Assume $p, q$ are complete sets of formulas (say in the variable $\bx$, with $\ell (\bx) < \kappa$) over $B$ that do not split over $A$ and $M$ is $|A|^+$-saturated. If $p \rest M = q \rest M$, then $p = q$. 
\end{lem}
\begin{proof}
  Let $\phi (\bx, \bb) \in p$ with $\bb \in B$. We show $\phi (\bx, \bb) \in q$ and the converse is symmetric. By saturation\footnote{Note that we are really using saturation for types \emph{of length less than $\kappa$} here.}, find $\bb' \in M$ such that $\tp (\bb' / A) = \tp (\bb / A)$. Since $p$ does not split over $A$, $\phi (\bx, \bb') \in p$. Since $p \rest M = q \rest M$, $\phi (\bx, \bb') \in q$. Since again $q$ does not split, $\phi (\bx, \bb) \in q$.
\end{proof}

We would like to study when the average is a nonsplitting extension. This is the purpose of the next definition.

\begin{defin}
  $\BI$ is \emph{$\chi$-based on $A$} if for any $B$, $\Av_{\chi} (\BI / B)$ does not split over $A$.
\end{defin}

The next lemma tells us that any sequence is based on a set of small size.

\begin{lem}[IV.1.23(2) in \cite{shelahaecbook}]\label{average-based}
  If $\BI$ is a sequence and $\BJ \subseteq \BI$ has size at least $\chi$, then $\BI$ is $\chi$-based on $\BJ$. 
\end{lem}
\begin{proof}
  Let $B$ be a set. Let $p := \Av_{\chi} (\BI / B)$. Note that $p \subseteq \Av_{\chi} (\BJ / B)$. Let $\bb, \bb' \in B$ be such that $\tp (\bb / \BJ) = \tp (\bb' / \BJ)$. Assume $\phi (\bx, \bb) \in p$. Then since $p \subseteq \Av_\chi (\BJ / B)$, let $\ba \in \BJ$ be such that $\models \phi[\ba, \bb]$. Since $\ba \in \BJ$, $\models \phi[\ba, \bb']$. Since there are at least $\chi$-many such $\ba$'s, $\neg \phi (\bx, \bb') \notin p$.
\end{proof}

We know that \emph{at most} one of $\phi$, $\neg \phi$ is in the average. It is very desirable to have that \emph{exactly} one is in, i.e.\ that the average is a \emph{complete} type. This is the purpose of the next definition. Recall from the beginning of this section that $\BI$ always denotes a sequence of elements \emph{of the same arity less than $\kappa$}.

\begin{defin}[V.A.2.1 in \cite{shelahaecbook2}]
  A sequence $\BI$ is said to be \emph{$\chi$-convergent} if $|\BI| \ge \chi$ and for any set $A$, $\Av_{\chi} (\BI / A)$ is a complete type over $A$. That is, whenever $\phi (\bx)$ is a formula with $\ell (\bx)$ equal to the arity of all the elements of $\BI$, then we have that exactly one of $\phi$ or $\neg \phi$ is in $\Av_{\chi} (\BI / A)$.
\end{defin}

\begin{remark}[Monotonicity]\label{avg-monot}
  If $\BI$ is $\chi$-convergent, $\BJ \subseteq \BI$, and $|\BJ| \ge \chi' \ge \chi$, then for any set $A$, $\Av_{\chi} (\BI / A) = \Av_{\chi'} (\BJ / A)$. In particular, $\BJ$ is $\chi'$-convergent.
\end{remark}

Recall \cite[III.1.7(1)]{shelahfobook} that if $T$ is a first-order stable theory and $\BI$ is an infinite sequence of indiscernibles (in its monster model), then $\BI$ is $\aleph_0$-convergent. The proof relies heavily on the compactness theorem. We would like a replacement of the form ``if $\mathcal{N}$ has some stability and $\BI$ is nice, then it is convergent. The next result is key. It plays the same role as the ability to extract indiscernible subsequences in first-order stable theories.

\begin{fact}[The convergent set existence theorem: V.A.2.8 in \cite{shelahaecbook2}]\label{conv-thm}
  Let $\chi_0 \ge (|L| + 2)^{<\kappa}$ be such that $\mathcal{N}$ does not have the order property of length $\chi_0^+$. Let $\mu$ be an infinite cardinal such that $\mu = \mu^{\chi_0} + 2^{2^{\chi_0}}$.

  Let $\BI$ be a sequence with $|\BI| = \mu^+$. Then there is $\BJ \subseteq \BI$ of size $\mu^+$ which is $\chi_0$-convergent.
\end{fact}

However having to extract a subsequence every time is too much for us. One issue is with the cardinal arithmetic condition on $\mu$: what if we have a sequence of length $\mu^+$ when $\mu$ is a singular cardinal of low cofinality? We work toward proving a more constructive result: \emph{Morley} sequences (defined below) are always convergent. The parameters represent respectively a bound on the size of $A$, the degree of saturation of the models, and the length of the sequence. They will be assigned default values in Hypothesis \ref{midway-hyp}.

\begin{defin}\label{coherent-def}
  We say $\seq{\ba_i : i \in I} \smallfrown \seq{N_i : i \in I}$ is a \emph{$(\chi_0, \chi_1, \chi_2)$-Morley sequence for $p$ over $A$} if:
  
  \begin{enumerate}
  \item $\chi_0 \le \chi_1 \le \chi_2$ are infinite cardinals, $I$ is a linear order, $A$ is a set, $p (\bx)$ is a set of formulas with parameters and $\ell (\bx) < \kappa$, and there is $\alpha < \kappa$ such that for all $i \in I$, $\ba_i \in \fct{\alpha}{\mathcal{N}}$.
  \item For all $i \in I$, $A \subseteq |N_i|$ and $|A| < \chi_0$.
  \item $\seq{N_i : i \in I}$ is increasing, and each $N_i$ is $\chi_1$-saturated.
  \item For all $i \in I$, $\ba_i$ realizes\footnote{Note that $\dom p$ might be smaller than $N_i$.} $p \rest N_i$ and for all $j > i$ in $I$, $\ba_i \in \fct{\alpha}{N_{j}}$.
  \item $i < j$ in $I$ implies $\ba_i \neq \ba_j$.
  \item $|I| \ge \chi_2$.
  \item\label{coherence-cond} For all $i < j$ in $I$, $\tp (\ba_i / N_i) = \tp (\ba_j / N_i)$.
    \item For all $i \in I$, $\tp (\ba_i / N_i)$ does not split over $A$.
  \end{enumerate}

  When $p$ or $A$ is omitted, we mean ``for some $p$ or $A$''. We call $\seq{N_i : i \in I}$ the \emph{witnesses} to $\BI := \seq{\ba_i : i \in I}$ being Morley, and when we omit them we simply mean that $\BI \smallfrown \seq{N_i : i \in I}$ is Morley for some witnesses $\seq{N_i : i \in I}$.
\end{defin}
\begin{remark}[Monotonicity]\label{coherent-monot}
  Let $\seq{\ba_i : i \in I} \smallfrown \seq{N_i : i \in I}$ be $(\chi_0, \chi_1, \chi_2)$-Morley for $p$ over $A$. Let $\chi_0' \ge \chi_0$, $\chi_1' \le \chi_1$, and $\chi_2' \le \chi_2$. Let $I' \subseteq I$ be such that $|I'| \ge \chi_2'$, then $\seq{\ba_i : i \in I'} \smallfrown \seq{N_i : i \in I'}$ is $(\chi_0', \chi_1', \chi_2')$-Morley for $p$ over $A$.
\end{remark}
\begin{remark}
  By the proof of \cite[I.2.5]{shelahfobook}, a Morley sequence is indiscernible (this will not be used).
\end{remark}

The next result tells us how to build Morley sequences inside a given model:

\begin{lem}\label{coherent-existence}
  Let $A \subseteq |M|$ and let $\chi \ge (|L| + 2)^{<\kappa}$ be such that $|A| \le \chi$. Let $p \in \Ss^{\alpha} (M)$ be nonalgebraic (that is, $a_i \notin |M|$ for all $i < \alpha$ for any $\ba$ realizing $p$) such that $p$ does not split over $A$, and let $\mu > \chi$. If:

  \begin{enumerate}
  \item $M$ is $\mu^+$-saturated.
  \item $\mathcal{N}$ is stable in $\mu$.
  \end{enumerate}

  Then there exists $\seq{\ba_i : i < \mu^+} \smallfrown \seq{N_i : i < \mu^+}$ inside $M$ which is $(\chi^+, \chi^+, \mu^+)$-Morley for $p$ over $A$.
\end{lem}
\begin{proof}
  We build $\seq{\ba_i : i < \mu^+}$ and $\seq{N_i : i < \mu^+}$ increasing such that for all $i < \mu^+$:

  \begin{enumerate}
    \item $A \subseteq |N_0|$.
    \item $|N_i| \subseteq |M|$.
    \item $\|N_i\| \le \mu $.
    \item $N_i$ is $\chi^+$-saturated.
    \item $\ba_i \in \fct{\alpha}{N_{i + 1}}$.
    \item $\ba_i$ realizes $p \rest N_i$.
  \end{enumerate}

  This is enough by definition of a Morley sequence (note that for all $i < \mu^+$, $\ba_i \notin \fct{\alpha}{N_i}$ by nonalgebraicity of $p$, so $\ba_i \neq \ba_j$ for all $j < i$).

  This is possible: assume inductively that $\seq{\ba_j : j < i} \smallfrown \seq{N_j : j < i}$ has been defined. Pick $N_i \subseteq M$ which is $\chi^+$-saturated, has size $\le \mu$, and contains $A \cup \bigcup_{j < i} N_j$. Such an $N_i$ exists: simply build an increasing chain $\seq{M_k : k < \chi^+}$ with $M_0 := A \cup \bigcup_{j < i} N_j$, $\|M_k\| \le \mu$, and $M_k$ realizing all elements of $\Ss (\bigcup_{k' < k} M_{k'})$ (this is where we use stability in $\mu$). Then $N_i := \bigcup_{k < \chi^+} M_k$ is as desired (we are using Hypothesis \ref{sat-hyp} to deduce that it is $\chi^+$-saturated for types of length less than $\kappa$). Now pick $\ba_i \in \fct{\alpha}{M}$ realizing $p \rest N_i$ (exists by saturation of $M$). 
\end{proof}

Before proving that Morley sequences are convergent (Theorem \ref{coherence-convergence}), we prove several useful lemmas:

\begin{lem}\label{coherent-avg}
  Let $\BI := \seq{\ba_i : i \in I}$ be $(\chi_0, \chi_1, \chi)$-Morley, as witnessed by $\seq{N_i : i \in I}$. Let $i \in I$ be such that $[j, \infty)_I$ has size at least $\chi$. Then $\Av_\chi (\BI / N_i) \subseteq \tp (\ba_i / N_i)$.
\end{lem}
\begin{proof}
  Let $\phi (\bx)$ be a formula over $N_i$ with $\ell (\bx) = \ell (\ba_i)$. Assume $\phi (\bx) \in \Av_\chi (\BI / N_i)$. By definition of average and assumption there exists $j \in [i, \infty)$ such that $\models \phi[\ba_j]$. By (\ref{coherence-cond}) in Definition \ref{coherent-def}, $\models \phi[\ba_i]$ so $\phi (\bx) \in \tp (\ba_i / N_i)$.
\end{proof}

\begin{lem}\label{lo-cut}
  Let $I$ be a linear order and let $\chi < |I|$ be infinite. Then there exists $i \in I$ such that both $(-\infty, i]_I$ and $[i, \infty)_I$ have size at least $\chi$.
\end{lem}
\begin{proof}
  Without loss of generality, $|I| = \chi^+$. Let $I_0 := \{i \in I \mid |(-\infty, i]_I| < \chi\}$ and let $I_1 := \{i \in I \mid |[i, \infty)_I| < \chi\}$. Assume the conclusion of the lemma fails. Then $I_0 \cup I_1 = I$. Thus either $|I_0| = \chi^+$ or $|I_1| = \chi^+$. Assume that $|I_0| = \chi^+$, the proof in case $|I_1| = \chi^+$ is symmetric. Let $\delta := \cf{I_0}$ and let $\seq{a_\alpha \in I_0 : \alpha < \delta}$ be a cofinal sequence. If $\delta < \chi^+$, then, since $I_0 = \cup_{\alpha<\delta} (-\infty, a_\alpha]_I$ has size $\chi^+$, there is $\alpha < \delta$ such that $|(-\infty, a_\alpha)_I| = \chi^+$.  If $\delta \geq \chi^+$, then $|(-\infty, a_\chi)_I| \geq \chi$.  Either of these contradict the definition of $I_0$.
\end{proof}

\begin{lem}\label{morley-based}
  Let $\BI$ be $(\chi^+, \chi^+, \chi^+)$-Morley over $A$ (for some type). If $\BI$ is $\chi$-convergent, then $\BI$ is $\chi$-based on $A$.
\end{lem}
\begin{proof}
  Let $\BI := \seq{\ba_i : i \in I}$ and let $\seq{N_i : i \in I}$ witness that $\BI$ is $\chi$-Morley over $A$. By assumption, $|I| \ge \chi^+$, so let $i \in I$ be as given by Lemma \ref{lo-cut}: both $(-\infty, i]_I$ and $[i, \infty)_I$ have size at least $\chi$. By Lemma \ref{average-based} and the definition of $i$, we can find $A' \subseteq |N_i|$ containing $A$ of size at most $\chi$ such that $\BI$ is $\chi$-based on $A'$.

  Let $p := \Av_{\chi} (\BI / \mathcal{N})$. Assume for a contradiction that $p$ splits over $A$ and pick witnesses such that $\phi (\bx, \bb)$, $\neg \phi (\bx, \bb') \in p$ and $\tp(\bb/A') = \tp(\bb'/A')$. Note that $p \rest N_i = \tp (\ba_i / N_i)$ by convergence and Lemma \ref{coherent-avg}. Since $N_i$ is $\chi^+$-saturated, we can find $\bb'' \in \fct{<\kappa}{|N_i|}$ such that $\tp (\bb'' / A') = \tp (\bb / A')$. Now either $\phi (\bx, \bb'') \in p$ or $\neg \phi (\bx, \bb'') \in p$. If $\phi (\bx, \bb'') \in p$, then $\phi (\bx, \bb''), \neg \phi (\bx, \bb')$ witness that $p$ splits over $A$ and if $\neg \phi (\bx, \bb'') \in q$, then $\phi (\bx, \bb)$, $\neg \phi (\bx, \bb'')$ witness the splitting. Either way, we can replace $\bb$ or $\bb'$ by $\bb''$. So (swapping the role of $\bb$ and $\bb'$ if necessary), assume without loss of generality that $\bb'' = \bb$ (so $\bb \in \fct{<\kappa}{|N_i|}$).

  By definition of a Morley sequence, $p \rest N_i$ does not split over $A$, so $\bb' \notin \fct{<\kappa}{|N_i|}$. Let $p_i' := p \rest N_i \cup \{\phi (\bx, \bb), \phi (\bx, \bb')\}$. We claim that $p_i'$ does not split over $A$: if it does, since $\phi (\bx, \bb')$ is the only formula of $p_i'$ with parameters outside of $N_i$, the splitting must be witnessed by $\phi (\bx, \bc)$, $\neg \phi (\bx, \bc')$, and one of them must be outside $N_i$, so $\bc = \bb'$. Now $\tp (\bb / A) = \tp (\bb' / A) = \tp (\bc' / A)$, and we have $\bc' \in \fct{<\kappa}{|N_i|}$ so by nonsplitting of $p \rest N_i$, also $\neg \phi (\bx, \bb) \in p \rest N_i$. This is a contradiction since we know $\phi (\bx, \bb) \in p \rest N_i$.

  Now, since $\BI$ is $\chi$-based on $A'$, $p$ does not split over $A'$ and by monotonicity $p_i'$ also does not split over $A'$. Now use the proof of Lemma \ref{ns-uq} (with $M = N_i$) to get a contradiction.
\end{proof}

We are now ready to prove the relationship between Morley and convergent:

\begin{thm}\label{coherence-convergence}
  Let $\chi_0 \ge (|L| + 2)^{<\kappa}$ be such that $\mathcal{N}$ does not have the order property of length $\chi_0^+$. Let $\chi := \left(2^{2^{\chi_0}}\right)^+$.

  If $\BI$ is a $(\chi_0^+, \chi_0^+, \chi)$-Morley sequence, then $\BI$ is $\chi$-convergent.
\end{thm}
\begin{proof}
  Write $\BI = \seq{\ba_i : i \in I}$ and let $\seq{N_i : i \in I}$ witness that it is Morley for $p$ over $A$.

Assume for a contradiction that $\BI$ is \emph{not} $\chi$-convergent. Then there exists a formula $\phi (\bx)$ (over $\mathcal{N}$) and linear orders $I_\ell \subseteq I$, $\ell = 0, 1$ such that $|I_\ell| = \chi$ and $i \in I_\ell$ implies\footnote{Where $\phi^0$ stands for $\phi$, $\phi^1$ for $\neg \phi$.} $\models \phi^\ell[\ba_i]$. By Fact \ref{conv-thm}, we can assume without loss of generality that $\BI_\ell := \seq{\ba_i : i \in I_\ell}$ is $\chi_0$-convergent. By Lemma \ref{morley-based} (with $\chi_0$ here standing for $\chi$ there), $\BI_\ell$ is $\chi_0$-based on $A$ for $\ell = 0,1$. Let $p_\ell := \Av_{\chi_0} (\BI_\ell / \mathcal{N})$. Since $\BI_\ell$ is $\chi_0$-based on $A$, $p_\ell$ does not split over $A$. By Lemma \ref{lo-cut}, pick $i_\ell \in I$ so that $|(i, \infty)_{I_\ell}| \ge \chi_0$ for $\ell = 0, 1$. let $i := \min (i_0, i_1)$. By Lemma \ref{coherent-avg} and convergence, $p_\ell \rest N_{i_\ell} = \tp (a_{i_\ell} / N_{i_\ell})$ so $p_\ell \rest N_i = \tp (a_{i_{\ell}} / N_i) = \tp (a_i / N_i)$, so $p_0 \rest N_i = p_1 \rest N_i$. By assumption, $N_i$ is $\chi_0^+$-saturated. By uniqueness for nonsplitting (Lemma \ref{ns-uq}), $p_0 = p_1$. However $\phi (\bx) \in p_0$ while $\neg \phi (\bx) \in p_1$, contradiction.
\end{proof}

From now on we assume:

\begin{hypothesis}\label{midway-hyp} \
  \begin{enumerate}
    \item $\chi_0 \ge (|L| + 2)^{<\kappa}$ is an infinite cardinal.
    \item $\mathcal{N}$ does not have the order property of length $\chi_0^+$.
    \item $\chi := \left(2^{2^{\chi_0}}\right)^{+}$.
    \item The default parameters for Morley sequences are $(\chi_0^+, \chi^+, \chi^+)$, and the default parameter for  averages and convergence is $\chi$. That is, Morley means $(\chi_0^+, \chi^+, \chi^+)$-Morley, convergent means $\chi$-convergent, $\Av (\BI / A)$ means $\Av_{\chi} (\BI / A)$, and based means $\chi$-based.
  \end{enumerate}
\end{hypothesis}

Note that Theorem \ref{coherence-convergence} and Hypothesis \ref{midway-hyp} imply that any Morley sequence is convergent. Moreover by Lemma \ref{morley-based}, any Morley sequence over $A$ is based on $A$. We will use this freely.

Before studying chains of saturated models, we generalize Lemma \ref{morley-based} to independence notions that are very close to splitting (the reason has to do with the translation to AECs (Section \ref{mainthm-sec})):

\begin{defin}\label{splitting-like-def}
  A \emph{splitting-like notion} is a binary relation $R (p, A)$, where $p \in \Ss^{<\infty} (B)$ for some set $B$ and $A \subseteq B$, satisfying the following properties:

  \begin{enumerate}
    \item Monotonicity: If $A \subseteq A' \subseteq B_0 \subseteq B$, $p \in \Ss^{<\infty} (B)$, and $R (p, A)$, then $R (p \rest B_0, A')$.
    \item Weak uniqueness: If $A \subseteq |M| \subseteq B$, $M$ is $\left(|A| + (|L| + 2)^{<\kappa}\right)^+$-saturated, and for $\ell = 1,2$, $q_\ell \in \Ss^{<\infty} (B)$, $R (q_\ell, A)$, and $q_1 \rest M = q_2 \rest M$, then $q_1 = q_2$.
    \item $R$ extends nonsplitting: If $p \in \Ss^{<\infty} (B)$ does not split over $A \subseteq B$, then $R (p, A)$.
  \end{enumerate}

  We also say ``$p$ does not $R$-split over $A$'' instead of $R (p, A)$.
\end{defin}
\begin{remark}
  If $R (p, A)$ holds if and only if $p$ does not split over $A$, then $R$ is a splitting-like notion: monotonicity is easy to check and $R$ is nonsplitting. Weak uniqueness is Lemma \ref{ns-uq}.
\end{remark}

\begin{lem}\label{av-eq-monot}
  Let $R$ be a splitting-like notion. Let $p \in \Ss^{<\kappa} (B)$ be such that $p$ does not $R$-split over $A \subseteq B$ with $|A| \le \chi_0$.

  Let $\BI := \seq{\ba_i : i \in I} \smallfrown \seq{N_i : i \in I}$ be Morley for $p$ over $A$.

  If $|\bigcup_{i \in I} N_i| \subseteq B$, then $\Av (\BI / B) = p$.
\end{lem}
\begin{proof}
  Since $\BI$ is Morley, $\BI$ is convergent. By Lemma \ref{morley-based}, $\BI$ is based on $A$. Thus we have that $\Av (\BI / B)$ does not split over $A$, so it does not $R$-split over $A$. Let $i \in I$ be such that $|(i, \infty)_I| \ge \chi$ (use Lemma \ref{lo-cut}). Then $\Av (\BI / N_i) = \tp (\ba_i / N_i) = p \rest N_i$ by Lemma \ref{coherent-avg}. By the weak uniqueness axiom of splitting-like relations (with $N_i$ here standing for $M$ there), $\Av (\BI / B) = p$.
\end{proof}

To construct Morley sequences, we will also use:

\begin{fact}\label{more-shelah-facts} \
  \begin{enumerate}
    \item\label{stab-1} If $\mu = \mu^{\chi_0} + 2^{2^{\chi_0}}$, then $\mathcal{N}$ is $(<\kappa)$-stable in $\mu$.
    \item\label{ns-2} Let $M$ be $\chi_0^+$-saturated. Then for any $p \in \Ss^{<\kappa} (M)$, there exists $A \subseteq |M|$ of size at most $\chi_0$ such that $p$ does not split over $A$.
  \end{enumerate}
\end{fact}
\begin{proof}
  The first result is \cite[V.A.1.19]{shelahaecbook2}. The second follows from \cite[V.A.1.12]{shelahaecbook2}: one only has to observe that the condition between $M$ and $\mathcal{N}$ there holds when $M$ is $\chi_0^+$-saturated.
\end{proof}

We can now get a (completely local) result on unions of saturated models.

\begin{thm} \label{union-sat-av}
  Assume:

  \begin{enumerate}
    \item $\lambda > \chi^+$ is such that $\mathcal{N}$ is stable in $\mu$ for unboundedly many $\mu < \lambda$.
    \item $\seq{M_i : i < \delta}$ is increasing and for all $i < \delta$, $M_i$ is $\lambda$-saturated. Write $M_\delta := \bigcup_{i < \delta} M_i$.
    \item\label{lc-hyp} For any $q \in \Ss (M_\delta)$, there exists a splitting-like notion $R$, $i < \delta$ and $A \subseteq |M_i|$ of size at most $\chi_0$ such that $q$ does not $R$-split over $A$.
  \end{enumerate}

  Then $M_\delta$ is $\lambda$-saturated.
\end{thm}
\begin{proof}
  By Hypothesis \ref{sat-hyp}, it is enough to check that $M_\delta$ is $\lambda$-saturated for types of length one. Let $p \in \Ss (B)$, $B \subseteq |M_\delta|$ have size less than $\lambda$. Let $q$ be an extension of $p$ to $\Ss (M_\delta)$. If $q$ is algebraic, then $p$ is realized inside $M_\delta$ so assume without loss of generality that $q$ is not algebraic. By assumption, there exists a splitting-like notion $R$, $i < \delta$ and $A \subseteq |M_i|$ such that $p$ does not $R$-split over $A$ and $|A| \le \chi_0$. Without loss of generality, $i = 0$. Now $M_0$ is $\chi_0^+$-saturated so (by Fact \ref{more-shelah-facts}) there exists $A' \subseteq |M_0|$ of size at most $\chi_0$ such that $q \rest M_0$ does not split over $A'$. By making $A$ larger if necessary, we can assume $A = A'$.

  Pick $\mu < \lambda$ such that $\mu \ge \chi^+ + |B|$ and $\mathcal{N}$ is stable in $\mu$. Such a $\mu$ exists by the hypothesis on $\lambda$. By Lemma \ref{coherent-existence}, there exists a sequence $\BI$ of length $\mu^+$ which is Morley for $q \rest M_0$ over $A$, with the witnesses living inside $M_0$. Thus $\BI$ is also Morley for $q$ over $A$.

  By Lemma \ref{av-eq-monot}, $\Av (\BI / M_\delta) = q$, and so in particular $\Av (\BI / B) = q \rest B = p$. By Lemma \ref{average-realize}, $p$ is realized by an element of $\BI \subseteq |M_0| \subseteq |M_\delta|$, as needed.
\end{proof}

The condition (\ref{lc-hyp}) in Theorem \ref{union-sat-av} is useful in case we know that the local character cardinal for \emph{chains} $\clc{\alpha}$ is significantly lower than the local character cardinal for \emph{sets} $\slc{\alpha}$. This is the case when a superstability-like condition holds. If we do not care about the local character cardinal for chains, we can state a version of Theorem \ref{union-sat-av} without condition (\ref{lc-hyp}).

\begin{cor}
  Assume:

  \begin{enumerate}
    \item\label{cor-stab-cond} $\lambda > \chi^+$ is such that $\mu^{\chi_0} < \lambda$ for all $\mu < \lambda$.
    \item $\seq{M_i : i < \delta}$ is increasing and for all $i < \delta$, $M_i$ is $\lambda$-saturated.
  \end{enumerate}

  If $\cf{\delta} \ge \chi_0^+$, then $\bigcup_{i < \delta} M_i$ is $\lambda$-saturated.
\end{cor}
\begin{proof}
  Fix $\alpha < \kappa$.  By Fact \ref{more-shelah-facts}(\ref{stab-1}), $\mathcal{N}$ is $\alpha$-stable in $\mu$ for any $\mu < \lambda$ with $\mu^{\chi_0} = \mu$ and $\mu \ge \chi$. By hypothesis, there are unboundedly many such $\mu$'s.
  
  Let $M_\delta := \bigcup_{i < \delta} M_i$. By an easy argument using the cofinality condition on $\delta$, $M_\delta$ is $\chi_0^+$-saturated. By Fact \ref{more-shelah-facts}(\ref{ns-2}), for any $p \in \Ss^{<\kappa} (M_\delta)$, there exists $A \subseteq |M_\delta|$ of size $\le \chi_0$ such that $p$ does not split over $A$. By the cofinality assumption on $\delta$, we can find $i < \delta$ such that $A \subseteq |M_i|$. Now apply Theorem \ref{union-sat-av} and get the result.
\end{proof}
\begin{remark}
  The proof shows that we can still replace (\ref{cor-stab-cond}) with ``$\lambda > \chi^+$ is such that $\mathcal{N}$ is stable in $\mu$ for unboundedly many $\mu < \lambda$''.
\end{remark}

We end this section with the following interesting variation: the cardinal arithmetic condition on $\lambda$ is improved, and we do not even need that the $M_i$'s be $\lambda$-saturated, only that they realize enough types from the previous $M_j$'s.

\begin{thm}\label{more-on-chain}
  Assume:

  \begin{enumerate}
    \item $\lambda > \chi$ is such that $\mu^{<\kappa} < \lambda$ for all $\mu < \lambda$ (or such that $\mathcal{N}$ is stable in $\mu$ for unboundedly many $\mu < \lambda$).
    \item\label{q-cond} $M$ is such that for any $q \in \Ss (M)$ there exists $\seq{M_i : i < \delta}$ strictly increasing so that:
      \begin{enumerate}
        \item $\delta \ge \lambda$ is a limit ordinal.
        \item $M = \bigcup_{i < \delta} M_i$
        \item For all $i < \delta$, $M_i$ is $\chi^+$-saturated and $M_{i + 1}$ realizes $q \rest M_i$.
        \item There exists a splitting-like notion $R$, $i < \delta$ and $A \subseteq |M_i|$ of size at most $\chi_0$ such that $q$ does not $R$-split over $A$.
      \end{enumerate}
  \end{enumerate}

  Then $M$ is $\lambda$-saturated.
\end{thm}
\begin{proof}
  By Hypothesis \ref{sat-hyp}, it is enough to check that $M$ is $\lambda$-saturated for types of length one. Let $p \in \Ss (B)$, $B \subseteq |M|$ have size less than $\lambda$. Let $q$ be an extension of $p$ to $\Ss (M)$. If $q$ is algebraic, then $p$ is realized inside $M$, so assume $q$ is not algebraic. Let $\seq{M_i : i < \delta}$ be as given by (\ref{q-cond}) for $q$. Let $R$ be a splitting-like notion for which there is $i < \delta$ and $A \subseteq |M_i|$ such that $q$ does not $R$-split over $A$ and $|A| \le \chi_0$. Without loss of generality, $i = 0$.

  Let $\mu := \left(\chi + |B|\right)^{<\kappa}$ (or take $\mu < \lambda$ such that $\mu \ge \chi + |B|$ and $\mathcal{N}$ is stable in $\mu$). Note that $\mu < \lambda$. For $i < \mu^+$, let $a_i \in |M_{i + 1}|$ realize $q \rest M_i$. By cofinality considerations, $\bigcup_{i < \mu^+} M_i$ is $\chi^+$-saturated. By Fact \ref{more-shelah-facts}, there exists $i < \mu^+$ and $A' \subseteq |M_i|$ such that $q \rest \bigcup_{i < \mu^+} M_i$ does not split over $A'$. By some renaming we can assume without loss of generality that $A' = A$. It is now easy to check that $\BI := \seq{a_i : i < \mu^+}$ is Morley for $q$ over $A$, as witnessed by $\seq{M_i : i < \mu^+}$.

  By Lemma \ref{av-eq-monot}, $\Av (\BI / M) = q$, and so in particular $\Av (\BI / B) = q \rest B = p$. By Lemma \ref{average-realize}, $p$ is realized by an element of $\BI \subseteq |M_0| \subseteq |M|$, as needed. 
\end{proof}



\section{Translating to AECs} \label{mainthm-sec}

To translate the result of the previous section to AECs, we will use the \emph{Galois Morleyization} of an AEC, a tool introduced in \cite{sv-infinitary-stability-afml}: Essentially, we expand the language of the AEC with a symbol for each Galois type. With enough tameness, Galois types then become syntactic.

\begin{defin}[3.3 in \cite{sv-infinitary-stability-afml}] \label{def-galois-m}
Let $K$ be an AEC and let $\kappa$ be an infinite cardinal. Define an (infinitary) expansion $\bigL$ of $L (K)$ by adding a relation symbol $R_p$ of arity $\ell (p)$ for each $p \in \gS^{<\kappa} (\emptyset)$. Expand each $N \in K$ to a $\bigL$-structure $\bigN$ by specifying that for each $\ba \in \bigN$, $R_p^{\bigN} (\ba)$ holds exactly when $\gtp (\ba / \emptyset; N) = p$. We write $\bigKp{\kappa}$ for $\bigK$. We call $\bigKp{\kappa}$ the \emph{$(<\kappa)$-Galois Morleyization} of $K$.
\end{defin} 
\begin{remark}\label{bigk-size}
  Let $K$ be an AEC and $\kappa$ be an infinite cardinal. Then $|L (\bigKp{\kappa})| \le |\gS^{<\kappa} (\emptyset)| + |L| \le 2^{<(\kappa + \LS (K)^+)}$.
\end{remark}

\begin{fact}[3.16 in \cite{sv-infinitary-stability-afml}]\label{galois-transl}
  Let $K$ be a $(<\kappa)$-tame AEC, and let $M \lea N_\ell$, $a_\ell \in |N_\ell|$, $\ell = 1,2$. Then $\gtp (a_1 / M; N_1) = \gtp (a_2 / M; N_2)$ if and only if\footnote{Recall that $\tp_{q\bigL_{\kappa, \kappa}}$ stands for quantifier-free $L_{\kappa, \kappa}$-type.} $\tp_{q\bigL_{\kappa, \kappa}} (a_1 / M; \widehat{N_1}) = \tp_{q\bigL_{\kappa, \kappa}} (a_2 / M; \widehat{N_2})$.

  Moreover the left to right direction does not need tameness: if $M \lea N_\ell$, $\ba_\ell \in \fct{<\infty}{|N_\ell|}$, $\ell = 1,2$, and $\gtp (\ba_1 / M; N_1) = \gtp (\ba_2 / M; N_2)$, then $\tp_{q\bigL_{\kappa, \kappa}} (\ba_1 / M; \widehat{N_1}) = \tp_{q\bigL_{\kappa, \kappa}} (\ba_2 / M; \widehat{N_2})$.
\end{fact}

Note that this implies in particular that (if $K$ is $(<\kappa)$-tame and has amalgamation) the Galois version of saturation and stability coincide with their syntactic analog in $\bigKp{\kappa}$. There is also a nice correspondence between the syntactic version of the order property defined at the beginning of Section \ref{averages-sec} and Shelah's semantic version \cite[4.3]{sh394}:

\begin{defin}\label{def-op}
  Let $\alpha$ and $\mu$ be cardinals and let $K$ be an AEC. A model $M \in K$ has the \emph{$\alpha$-order property of length $\mu$} if there exists $\seq{\ba_i : i < \mu}$ inside $M$ with $\ell (\ba_i) = \alpha$ for all $i < \mu$, such that for any $i_0 < j_0 < \mu$ and $i_1 < j_1 < \mu$, $\gtp (\ba_{i_0} \ba_{j_0} / \emptyset; N) \neq \gtp (\ba_{j_1} \ba_{i_1} / \emptyset; N)$.

  $M$ has the \emph{$(<\alpha)$-order property of length $\mu$} if it has the $\beta$-order property of length $\mu$ for some $\beta < \alpha$. $M$ has the \emph{order property of length $\mu$} if it has the $\alpha$-order property of length $\mu$ for some $\alpha$.

  \emph{$K$ has the $\alpha$-order of length $\mu$} if some $M \in K$ has it. \emph{$K$ has the  order property} if it has the order property for every length.
\end{defin}

\begin{fact}[4.4 in \cite{sv-infinitary-stability-afml}]\label{op-transl}
    Let $K$ be an AEC. Let $\bigK := \bigKp{\kappa}$. If $\bigN \in \bigK$ has the (syntactic) order property of length $\chi$, then $N$ has the (Galois) $(<\kappa)$-order property of length $\chi$. Conversely, if $\chi \ge 2^{<(\kappa + \LS (K)^+)}$ and $N$ has the (Galois) $(<\kappa)$-order property of length $(2^{\chi})^+$, then $\bigN$ has the (syntactic) order property of length $\chi$.
\end{fact}

We will use Facts \ref{galois-transl} and \ref{op-transl} freely in this section. We will also use the following results about stability and the order property:

\begin{fact}[4.5 and 4.13 in \cite{sv-infinitary-stability-afml}]\label{op-facts} 
  Let $K$ be an $(<\kappa)$-tame AEC with amalgamation. The following are equivalent:
  \begin{enumerate}
  \item $K$ is stable in some $\lambda \ge \kappa + \LS (K)$.
  \item\label{op-facts-2b} There exists $\mu \le \lambda_0 < \hanfs{\kappa + \LS (K)^+}$ (see Notation \ref{set-thy-notation}) such that $K$ is stable in any $\lambda \ge \lambda_0$ with $\lambda = \lambda^{<\mu}$.
  \item $K$ does not have the order property.
  \item There exists $\chi < \hanfs{\kappa + \LS (K)^+}$ such that $K$ does not have the $(<\kappa)$-order property of length $\chi$.
  \end{enumerate}
\end{fact}

It remains to find an independence notion to satisfy condition (\ref{lc-hyp}) in Theorem \ref{union-sat-av}. The splitting-like notion $R$ there will be given by the following:

\begin{defin}
  Let $K$ be an AEC and let $\kappa$ be an infinite cardinal. For $p \in \gS^{<\infty} (B; N)$ and $A \subseteq B$, say $p$ \emph{$\kappa$-explicitly does not split over $A$} if whenever $p = \gtp (\bc / B; N)$, for any $\bb, \bb' \in \fct{<\kappa}{B}$, if $\gtp (\bb / A; N) = \gtp (\bb' / A; N)$, then $\tp_{q\bigL_{\kappa, \kappa}} (\bc \bb / A; N) = \tp_{q\bigL_{\kappa, \kappa}} (\bc \bb' / A; N)$, where $\bigL = L (\bigKp{\kappa})$.
\end{defin}
\begin{remark}
  This is closely related to explicit nonsplitting defined in \cite[3.13]{bgkv-apal}. The definition there is that $p$ explicitly does not split if and only if it $\kappa$-explicitly does not split for all $\kappa$. When $K$ is fully $(<\kappa)$-tame and short (see \cite[3.3]{tamelc-jsl}), this is equivalent to just asking for $p$ to $\kappa$-explicitly not split.
\end{remark}
\begin{remark}[Syntactic invariance]\label{syn-invariance}
  Let $\bigK := \bigKp{\kappa}$. Assume $\tp_{q\bigL_{\kappa, \kappa}} (\bc / B; N) = \tp_{q\bigL_{\kappa, \kappa}} (\bc' / B; N)$ and $\gtp (\bc / B; N)$ $\kappa$-explicitly does not split over $A \subseteq B$. Then $\gtp (\bc' / B; N)$ $\kappa$-explicitly does not split over $A$.
\end{remark}

We will use the following definition of an independence relation, which appears implicitly in \cite[4.8]{ss-tame-jsl}.

\begin{defin}
  Let $K$ be an AEC with amalgamation and let $\lambda \ge \LS (K)$ be such that $\K$ is $\lambda$-tame and stable in $\lambda$. For $M \lea N$ with $M$ $\lambda^+$-saturated, we say that $p \in \gS (N)$ \emph{does not $\lambda$-fork over $M$} if there exists $M_0 \in \K_\lambda$ such that $M_0 \lea M$ and $p$ does not $\lambda$-split over $M_0$ (that is \cite[I.3.2]{sh394}, whenever $N_\ell$, $\ell = 1,2$, are of size $\lambda$ such that $M_0 \lea N_\ell \lea N$ and $f: N_1 \cong_{M_0} N_2$, we have that $f(p \rest N_1) = p \rest N_2$). We write $a \nf_M^{\lambdanf} N$ to say that $\gtp (a / N)$ does not $\lambda$-fork over $M$ and will apply the definition of the properties Definition \ref{loc-card-def} and Fact \ref{coheir-props-fact} to it.
\end{defin}

\begin{fact}[\S4 and \S5 in \cite{ss-tame-jsl}]\label{lambda-forking-props}
  Let $K$ be an AEC with amalgamation and let $\lambda \ge \LS (K)$ be such that $\K$ is $\lambda$-tame, stable in $\lambda$, and has no maximal models in $\lambda$.

  Then $\lambda$-nonforking satisfies invariance, monotonicity, transitivity (i.e.\ if $M_1 \lea M_2 \lea M_3$ are such that $M_1$ and $M_2$ are $\lambda^+$-saturated, $p \in \gS (M_3)$, $p$ does not $\lambda$-fork over $M_2$, $p \rest M_2$ does not $\lambda$-fork over $M_1$, then $p$ does not $\lambda$-fork over $M_1$), and uniqueness. Moreover $\slc{1} (\nf^{\lambdanf} \rest \Ksatp{\lambda^+}) = \lambda^{++}$.
\end{fact}

We recall the definition of superstability from \cite[10.1]{indep-aec-apal} using local character of nonsplitting. Note that it coincides with the first-order definition (see \cite[10.9]{indep-aec-apal}) and is equivalent to the definition implicit in \cite{gvv-mlq, ss-tame-jsl} and explicit in \cite[7.12]{grossberg2002}.

\begin{defin}[Superstability]\label{ss-def}
  Let $K$ be an AEC.

  \begin{enumerate}
    \item For $M, N \in K$, say $N$ is \emph{universal over $M$} if and only if $M \lea N$ and whenever we have $M' \gea M$ such that $\|M'\| = \|M\|$, then there exists $f: M' \xrightarrow[M]{} N$.
    \item $K$ is $\lambda$-superstable if:
      \begin{enumerate}
      \item $\LS (K) \le \lambda$ and $\K_\lambda \neq \emptyset$.
      \item $\K_\lambda$ has amalgamation, joint embedding, and no maximal models.
      \item $\K$ is stable in $\lambda$.
      \item $\K$ has no long splitting chains in $\lambda$: for any limit $\delta < \lambda^+$ and increasing continuous $\seq{M_i : i \le \delta}$ in $\K_\lambda$ with $M_{i + 1}$ universal over $M_i$ for all $i < \delta$, and any $p \in \gS (M_\delta)$, there exists $i < \delta$ such that $p$ does not $\lambda$-split over $M_i$.
      \end{enumerate}
  \end{enumerate}
\end{defin}

Note that superstability implies local character of $\lambda$-forking, and superstability transfers up assuming tameness:

\begin{fact}\label{ss-fact}
  Let $\K$ be an AEC with amalgamation that is $\lambda$-tame and $\lambda$-superstable.

  \begin{enumerate}
  \item \cite[4.11]{ss-tame-jsl} $\clc{1} (\nf^{\lambdanf} \rest \Ksatp{\lambda^+}) = \aleph_0$.
  \item $\K$ is $\lambda'$-superstable for all $\lambda' \ge \lambda$.
  \end{enumerate}
\end{fact}

The next result imitates \cite[5.6]{bgkv-apal}:

\begin{lem}\label{nesp-transl}
  Let $K$ be an AEC with amalgamation and let $\lambda \ge \LS (K)$ be such that $\K$ is $\lambda$-tame, stable in $\lambda$, and has no maximal models in $\lambda$. Let $\kappa \le \lambda^+$.

  Let $M \lea N$ be given with $M$ $\lambda^+$-saturated. Let $p \in \gS (N)$. If $p$ does not $\lambda$-fork over $M$, then $p$ $\kappa$-explicitly does not split over $A$.
\end{lem}
\begin{proof}
  By definition of $\lambda$-nonforking, there exists $M_0 \lea M$ of size $\lambda$ such that $p$ does not $\lambda$-split over $M_0$. We will show that $p$ explicitly does not $\kappa$-split over $M_0$ which is enough by base monotonicity of explicit $\kappa$-nonsplitting.

  Work inside a monster model $\sea$ and write $p = \gtp (c / N)$. Let $\bb, \bb' \in \fct{<\kappa}{|N|}$ be such that $\gtp (b / M_0) = \gtp (\bb' / M_0)$. Let $f$ be an automorphism of $\sea$ fixing $M_0$ such that $f (\bb) = \bb'$. By invariance, $f (p)$ does not $\lambda$-split over $M_0$. Now using uniqueness of $\lambda$-splitting (see \cite[I.4.12]{vandierennomax}), $f (p \rest M_0\bb) = p \rest M_0 \bb'$. The result follows.
\end{proof}

The next technical lemma captures the essence of our translation:

\begin{lem}\label{aec-technical}
  Let $K$ be a $(<\kappa)$-tame AEC with amalgamation and no maximal models. Let $\chi_0$ be such that:

  \begin{enumerate}
    \item $\chi_0 \ge 2^{< (\kappa + \LS (K)^+)}$.
    \item $K$ does not have the $(<\kappa)$-order property of length $\chi_0^+$.
  \end{enumerate}

  Set $\chi := \left(2^{2^{\chi_0}}\right)^+$. Let $\lambda$ be such that:

  \begin{enumerate}
    \item $\lambda > \chi^+$.
    \item $K$ is stable in $\mu$ for unboundedly many $\mu < \lambda$.
  \end{enumerate}

  Let $\theta := \clc{1} (\nf^{\chinf} \rest \Ksatp{\chi^+})$. Then:
  
  \begin{enumerate}
  \item If $\seq{M_i : i < \delta}$ is an increasing chain of $\lambda$-saturated models and $\cf{\delta} \ge \theta$, then $\bigcup_{i < \delta} M_i$ is $\lambda$-saturated.
  \item If $M \in K$ is such that for any $q \in \gS (M)$ there exists $\seq{M_i : i < \delta}$ strictly increasing so that:
      \begin{enumerate}
        \item $\delta \ge \lambda$ and $\cf{\delta} \ge \theta$.
        \item $M = \bigcup_{i < \delta} M_i$.
        \item For all $i < \delta$, $M_i$ is $\chi^+$-saturated and $M_{i + 1}$ realizes $q \rest M_i$.
      \end{enumerate}

      Then $M$ is $\lambda$-saturated.
  \end{enumerate}
\end{lem}
\begin{proof}
  We prove the first statement. The proof of the second is analogous but uses Theorem \ref{more-on-chain} instead of Theorem \ref{union-sat-av}. Set $M_\delta := \bigcup_{i < \delta} M_i$. Let $N \gea M_\delta$ be such that $N$ realizes all types in $\gS^{<\kappa} (M_\delta)$. We check that $M_\delta$ is $\lambda$-saturated in $N$. Let $\bigK := \bigKp{\kappa}$ be the $(<\kappa)$-Galois Morleyization of $K$. Let $\mathcal{N} := \bigN$. By $(<\kappa)$-tameness, it is enough to show that $\widehat{M_\delta}$ is (syntactically) $\lambda$-saturated in $\mathcal{N}$. Work inside $\mathcal{N}$ in the language of $\bigK$. We also let $\mathcal{S} := \{|M| \mid M \lea N\}$. Note that $\mathcal{S}$ satisfies Hypothesis \ref{model-hyp}.

  First observe that Hypothesis \ref{midway-hyp} holds as (Remark \ref{bigk-size}) $|L (\bigK)| \le 2^{<(\kappa + \LS (K)^+)}$, so $\chi_0$ has all the required properties. Also, Hypothesis \ref{sat-hyp} holds by \cite[II.1.14]{shelahaecbook}. Note that $\K$ is stable in $\chi$ by Fact \ref{more-shelah-facts}(\ref{stab-1}). By hypothesis, $\lambda > \chi^+$. We want to use Theorem \ref{union-sat-av}, and it remains to check that (\ref{lc-hyp}) there holds.

  For $A \subseteq B$ and $p \in \Ss^{<\infty} (B)$, define the relation $R (p, A)$ to hold if and only if $p = \tp (\bc / B)$ and $\gtp (\bc / B; N)$ $\kappa$-explicitly does not split over $A$. Note that this is well-defined by Remark \ref{syn-invariance}. We want to check that this is a splitting-like notion (Definition \ref{splitting-like-def}). By definition of $\kappa$-explicit nonsplitting, if $p \in \Ss^{<\infty} (B)$ does not split over $A \subseteq B$, then $R (p, A)$. Also, it is easy to check that $R$ satisfies the monotonicity axiom. It remains to check the weak uniqueness axiom. So let $M$ be $\mu := \left(|A| + (|L (\bigK)| + 2)^{<\kappa}\right)^+$-saturated, $A \subseteq |M| \subseteq B$, and for $\ell = 1,2$, $q_\ell \in \Ss^{<\infty} (B)$, $R (q_\ell, A)$ and $q_1 \rest M = q_2 \rest M$. Note that $M$ is also $\mu$-saturated in the Galois sense (by tameness and Remark \ref{lc-hyp}). Thus we can imitate the proof of Lemma \ref{ns-uq}, using Galois saturation instead of syntactic saturation to get $\bb'$ satisfying $\gtp (\bb' / A) = \gtp (\bb / A)$ (instead of just $\tp (\bb' / A) = \tp (\bb' / A)$ as there). The definition of $\kappa$-explicit nonsplitting then makes the proof go through.

  Now let $q \in \gS (M_\delta)$. By definition of $\theta$, there exists $i < \delta$ such that $q$ does not $\chi$-fork over $M_i$. Now by set local character, there exists $M \lea M_i$ of size $\chi^+$ such that $q \rest M_i$ does not $\chi$-fork over $M$. By transitivity, $q$ does not $\chi$-fork over $M$. By Lemma \ref{nesp-transl}, working syntactically inside $\mathcal{N}$, $q$ does not $R$-split over $M$. Thus (\ref{lc-hyp}) holds. Therefore $M_\delta$ is $\lambda$-saturated, as desired.
\end{proof}

We obtain the following result on chains of saturated models in stable AECs:

\begin{thm}\label{aec-average-stable}
  Let $K$ be a $(<\kappa)$-tame AEC with amalgamation, $\kappa \ge \LS (K)^+$. If $K$ is stable, then there exists $\chi_0 \le \lambda_0 < \hanfs{\kappa}$ (see Notation \ref{set-thy-notation}) satisfying the following property:

  If $\lambda \ge \lambda_0$ is such that $\mu^{\chi_0} < \lambda$ for all $\mu < \lambda$ (or just that $K$ is stable in $\mu$ for unboundedly many $\mu < \lambda$), then whenever  $\seq{M_i : i < \delta}$ is an increasing chain of $\lambda$-saturated models with $\cf{\delta} \ge \lambda_0$, we have that $\bigcup_{i < \delta} M_i$ is $\lambda$-saturated.
\end{thm}
\begin{proof}
  Using Fact \ref{op-facts}, pick $\chi_0 \le \mu_0 < \hanfs{\kappa}$ such that:

  \begin{enumerate}
    \item $\chi_{0}^+ \ge 2^{<(\kappa + \LS (K)^+)} + \kappa^+$.
    \item $K$ is stable in any $\mu \ge \mu_0$ with $\mu = \mu^{\chi_{0}}$.
    \item $K$ does not have the $(<\kappa)$-order property of length $\chi_{0}^+$.
  \end{enumerate}

  Now set $\lambda_0 := (2^{2^{\chi_0}})^{+3}$ and apply Lemma \ref{aec-technical}.
\end{proof}

The statement becomes much nicer in superstable AECs:

\begin{thm}\label{aec-average-superstable}
  Let $K$ be a $(<\kappa)$-tame AEC with amalgamation, $\LS (K)^+ \le \kappa$. Let $\mu \ge \LS (\K)$ be a cardinal with $\mu^+ \ge \kappa$ and assume that $\K$ is $\mu$-superstable. Then there exists $\lambda_0 < \hanfs{\kappa} + \mu^{++}$ with $\lambda_0 > \mu$ such that that for any $\lambda \ge \lambda_0$:
  \begin{enumerate}
    \item $\Ksatp{\lambda}$ is an AEC with $\LS (\Ksat) = \lambda$.
    \item\label{aec-average-superstable-2} If $M \in \Ksatp{\lambda_0}$ is such that for any $q \in \gS (M)$ there exists $\seq{M_i : i < \lambda}$ a resolution of $M$ in $\Ksatp{\lambda_0}$ such that  $q \rest M_i$ is realized in $M_{i + 1}$ for all $i < \lambda$, then $M \in \Ksatp{\lambda}$.
  \end{enumerate}
\end{thm}
\begin{proof} \
  \begin{enumerate}
    \item We first show that any increasing union of $\lambda$-saturated models is saturated. Let $\lambda_{00} < \hanfs{\kappa}$ be as given by the proof of Theorem \ref{aec-average-stable} and let $\lambda_0 := \lambda_{00} + \mu^+$. By Fact \ref{ss-fact}, $\clc{1} (\nf^{\nfp{\lambda_0}} \rest \Ksatp{\lambda_0}) = \aleph_0$. Now apply Lemma \ref{aec-technical} (note that by Fact \ref{ss-fact}, $K$ is stable in any $\mu' \ge \mu$). To see that $\LS (\Ksatp{\lambda}) = \lambda$, imitate the proof of \cite[Theorem III.3.12]{shelahfobook}.
    \item Similar: use the second conclusion of Lemma \ref{aec-technical}.
  \end{enumerate}
\end{proof}


\section{On superstability in AECs}\label{ss-def-sec}

In the introduction to \cite{shelahaecbook}, Shelah points out the importance of finding a definition of superstability for AECs. He also remarks (p.~19) that superstability in AECs suffers from ``schizophrenia'': definitions that are equivalent in the first-order case might not be equivalent in AECs. In this section, we point out that Definition \ref{ss-def} implies several other candidate definitions of superstability. Recall from Fact \ref{ss-fact} that Definition \ref{ss-def} implies that the class is stable on a tail of cardinals. We will focus on five other definitions:

\begin{enumerate}
  \item For every high-enough $\lambda$, the union of any increasing chain of $\lambda$-saturated models is $\lambda$-saturated. This is the focus of this paper and is equivalent to first-order superstability by \cite[Theorem 13]{agchains}.
  \item The existence of a saturated model of size $\lambda$ for every high-enough $\lambda$. In first-order, this is an equivalent definition of superstability by the saturation spectrum theorem (Fact \ref{sat-spectrum}).
  \item The existence of a superlimit model of size $\lambda$ for every high-enough $\lambda$. This is the definition of superstability listed by Shelah in \cite[N.2.4]{shelahaecbook}. Recall that a model $M \in K_\lambda$ is \emph{superlimit} if it is universal, has an isomorphic proper extension in $K_\lambda$, and whenever $\seq{M_i  : i < \delta}$ is increasing in $K_\lambda$, $\delta < \lambda^+$, and $M_i \cong M$ for all $i < \delta$, then $\bigcup_{i < \delta} M_i \cong M$.
  \item The existence of a good $\lambda$-frame on a subclass of saturated models (e.g.\ for every high-enough $\lambda$). Recall that a good frame is essentially a forking-like notion for types of length one (see \cite[II.2.1]{shelahaecbook} for the formal definition). Good frames are \emph{the} central notion in \cite{shelahaecbook} and are described by Shelah as a ``bare bone'' definition of superstability. 
  \item The uniqueness of limit models of size $\lambda$ for every high-enough $\lambda$: Recall that a model $M$ is \emph{$(\lambda, \delta)$-limit over $M_0$} if $M_0 \lea M$ are in $K_\lambda$, $\delta < \lambda^+$ is a limit ordinal and there exists $\seq{M_i : i \le \delta}$ increasing continuous such that $M_\delta = M$ and $i < \delta$ implies $M_i \ltu M_{i + 1}$ (recall Definition \ref{ss-def}). We say $K_\lambda$ has \emph{uniqueness of limit models} if for any $M_0 \in K_\lambda$, any limit $\delta_1, \delta_2 < \lambda^+$, any $M_\ell$ which are $(\lambda, \delta_\ell)$-limit over $M_0$ are isomorphic over $M_0$. Uniqueness of limit models is central in \cite{sh394, shvi635, vandierennomax, nomaxerrata} and is further examined in \cite{gvv-mlq} (Theorem 6.1 there proves that the condition is equivalent to first-order superstability). These papers all prove the uniqueness under a categoricity (or no Vaughtian pair) assumption. In \cite[II.4.8]{shelahaecbook}, uniqueness of limit models is proven from a good frame (see also \cite[9.2]{ext-frame-jml} for a detailed writeup). This is used in \cite{bg-v11-toappear} to get eventual uniqueness of limit models from categoricity, but the authors have to make an extra assumption (the extension property for coheir). 
\end{enumerate}

Note that some easy implications between these definitions are already known (see for example \cite[2.3.12]{drueckthesis}). We now show that assuming amalgamation and tameness, if $K$ is superstable, then all five of these conditions hold. This gives an eventual version of \cite[Conjecture 4.2.5]{drueckthesis}. It also shows how to build a good frame without relying on categoricity (as opposed to all previous constructions, see \cite[II.3.7]{shelahaecbook}, \cite[7.3]{ss-tame-jsl}, or \cite[10.16]{indep-aec-apal}).

\begin{thm}\label{ss-def-implication}
  If $K$ is a $\mu$-tame, $\mu$-superstable AEC with amalgamation, then there exists $\lambda_0 < \hanf{\mu}$ such that for all $\lambda \ge \lambda_0$:

  \begin{enumerate}
    \item\label{ss-def-1} The union of any increasing chain of $\lambda$-saturated models in $K$ is $\lambda$-saturated.
    \item\label{ss-def-2} $K$ has a saturated model of size $\lambda$.
    \item\label{ss-def-3} $K$ has a superlimit model of size $\lambda$.
    \item\label{ss-def-4} There exists a type-full good $\lambda$-frame with underlying class $\Ksatp{\lambda}_\lambda$.
    \item\label{ss-def-5} $K_\lambda$ has uniqueness of limit models.
  \end{enumerate}
\end{thm}
\begin{proof}
  Note that by Fact \ref{ss-fact}, $K_{\ge \mu}$ has no maximal models, joint embedding, and is stable in every cardinal. Let $\lambda_0 < \hanf{\mu}$ be as given by Theorem \ref{aec-average-superstable} and let $\lambda \ge \lambda_0$. Then $\Ksat$ is an AEC with $\LS (\Ksat) = \lambda$. Thus (\ref{ss-def-1}) and (\ref{ss-def-2}) hold. If $M$ is the saturated model of size $\lambda$, then it is easy to check that $M$ is superlimit: it is universal as $K_{\ge \mu}$ has joint embedding, it has a saturated proper extension of size $\lambda$ since $\LS (\Ksat) = \lambda$, and any increasing chain of saturated models in $K_\lambda$ of length less than $\lambda^+$ has a saturated union. Thus (\ref{ss-def-3}) holds. To see (\ref{ss-def-4}), use \cite[10.8(2c)]{indep-aec-apal}. 

  We are now ready to prove (\ref{ss-def-5}). As observed above, a good frame implies uniqueness of limit models. Thus $\Ksat_\lambda$ has uniqueness of limit models. It follows that $K_\lambda$ has uniqueness of limit models: Let $M_\ell$ be $(\lambda, \delta_\ell)$-limit over $M_0$, $\ell = 1,2$. Pick $M_0' \gea M_0$ in $\Ksat_\lambda$. By universality, $M_\ell$ is also $(\lambda, \delta_\ell)$-limit over some copy of $M_0'$, so after some renaming we can assume without loss of generality that $M_0 = M_0'$. For $\ell = 1,2$, build $\seq{M_i' : i \le \delta_\ell}$ increasing continuous such that for all $i < \delta_\ell$, $M_i' \in \Ksatp{\lambda}_\lambda$ and $M_i' \ltu M_{i + 1}'$. This is easy to do and by a back and forth argument, $M_\ell \cong_{M_0} M_{\delta_\ell}'$. By uniqueness of limit models in $\Ksat$, $M_{\delta_1}' \cong_{M_0} M_{\delta_2}'$. Composing the isomorphisms, we obtain that $M_1 \cong_{M_0} M_2$.
\end{proof}

\bibliographystyle{amsalpha}
\bibliography{union-saturated}

\end{document}